\documentclass[11pt]{article}    
    \usepackage{graphicx}
   \usepackage{epstopdf}
    \usepackage{graphicx,multirow}

\usepackage{wrapfig} 

\usepackage{amsmath,amssymb,amsthm,enumerate,graphicx,bbm,appendix,booktabs,enumitem,verbatim}
\usepackage{dsfont} 
\usepackage[dvipsnames,usenames]{color}
\usepackage[colorlinks,%
linkcolor=BrickRed,%
filecolor =BrickRed,%
citecolor=RoyalPurple,%
]{hyperref}

\usepackage{amsmath, amssymb, bbm, xspace}

\newcommand{\ie}{i.e.\@\xspace} 

\newcommand{\Real}{\ensuremath{\mathbb{R}}}

\def\disp{\displaystyle}

\def\half  {{\textstyle{1\over 2}}}

\def\spose#1{\hbox to 0pt{#1\hss}}

\def\text #1{\hbox{\quad#1\quad}}

\def\nthinsp{\mskip -2   mu}

\def\P{_{\scriptscriptstyle P}}

\def\superstar{^{\raise 0.5pt\hbox{$\nthinsp *$}}}
\def\SUPERSTAR{^{\raise 0.5pt\hbox{$*$}}}

\def\lamstarT {\lambda^{\raise 0.5pt\hbox{$\nthinsp *$}T}}

\def\Ascr{{\cal A}}

\def\Fscr{{\cal F}}

\def\Kscr{{\cal K}}

\def\Oscr{{\cal O}}

\def\hbar{\skew{4.2}\bar h}

		\def\bk1{{\rm 1\kern-.17em l}}
		\def\bkD{{\rm I\kern-.17em D}}
		\def\bkR{{\rm I\kern-.17em R}}
		\def\bkP{{\rm I\kern-.17em P}}
		\def\bkY{{\bf \kern-.17em Y}}
		\def\bkZ{{\bf \kern-.17em Z}}

		\def\beq{\begin{eqnarray}}
		\def\bc{\begin{center}}
		\def\be{\begin{enumerate}}
		\def\bi{\begin{itemize}}
		\def\bs{\begin{small}}
		\def\bS{\begin{slide}}
		\def\ec{\end{center}}
		\def\ee{\end{enumerate}}
		\def\ei{\end{itemize}}
		\def\es{\end{small}}
		\def\eS{\end{slide}}
		\def\eeq{\end{eqnarray}}

	\def\cp2problem#1#2#3#4{\fbox
		 {\begin{tabular*}{0.9\textwidth}
			{@{}l@{\extracolsep{\fill}}l@{\extracolsep{6pt}}l@{\extracolsep{\fill}}c@{}}
				#1 & & $#4 $
			\end{tabular*}}}

\newcommand{\pmat}[1]{\begin{pmatrix} #1 \end{pmatrix}}
		
		\renewcommand{\emph}[1]{\textbf{#1}}

		\def\bkE{{\mathbb{E}}}
		\def\bk1{{\rm 1\kern-.17em l}}
		\def\bkD{{\rm I\kern-.17em D}}
		\def\bkR{{\rm I\kern-.17em R}}
		\def\bkP{{\rm I\kern-.17em P}}
		
		\def\bkZ{{\bf{Z}}}

\newcommand {\beeq}[1]{\begin{equation}\label{#1}}
\newcommand {\eeeq}{\end{equation}}
\newcommand {\bea}{\begin{eqnarray}}
\newcommand {\eea}{\end{eqnarray}}

\def\texitem#1{\par\smallskip\noindent\hangindent 25pt
               \hbox to 25pt {\hss #1 ~}\ignorespaces}

\usepackage{amsmath} 
\usepackage{amssymb}  
\usepackage{color}
\usepackage{cite}
\usepackage{verbatim}

\def\P{\mathbf{P}}

\newtheorem{lemma}{Lemma}
\newtheorem{definition}{Definition}
\newtheorem{assumption}{Assumption}
\newtheorem{proposition}{Proposition}
\newtheorem{corollary}{Corollary}

\newtheorem{example}{Example}

\newcommand{\us}[1]{{\color{black}#1}}

 \newcommand{\remove}[1]{}

\def\Real{\mathbb{R}}

\def\Fscr{{\mathcal F}}

\usepackage{ulem}
\def\vvs#1{{\color{black}{#1}}}
\def\vs#1{{\color{black}{#1}}}
\def\ka#1{{\color{black}{#1}}}
\def\us#1{{\color{black}{#1}}}
\def\uvs#1{{\color{black}{#1}}}

\def\ask#1{{\color{black}{#1}}}

\begin{document}
\bibliographystyle{unsrt}
\title{Optimal stochastic extragradient schemes for  pseudomonotone
	stochastic variational inequality problems and their variants}
\author{Aswin Kannan and Uday V.~Shanbhag\thanks{\vs{The first author is at IBM research while the second is affiliated} with  Industrial and Manufacturing Engineering,
The Pennsylvania State University, University Park, PA, 16802
Email: \vs{\{aswinkannan1987@gmail.com,udaybag@psu.edu\}}. \vs{This research has been partially 
supported by} NSF Awards 1246887 (CAREER), \vs{1538193, and  1408366.}
}}
\date{\today}
\maketitle
\begin{abstract}
We consider the stochastic variational inequality problem in which the map is
expectation-valued in a component-wise sense.  Much of the available
convergence theory and rate statements for stochastic approximation schemes \vs{are}
limited to monotone maps. However, non-monotone stochastic variational
inequality problems are not uncommon and are seen to arise from product
pricing, fractional optimization problems, and  subclasses of economic
equilibrium problems. Motivated by the need to address a broader class of maps,
we make the following contributions: (i) We present an extragradient-based
stochastic approximation scheme and prove that the iterates converge to a
solution of the original problem under either pseudomonotonicity requirements
or a suitably defined acute angle condition. Such statements are shown to be
generalizable to the stochastic mirror-prox framework; (ii) Under strong
pseudomonotonicity, we show that the mean-squared error
in the solution iterates produced by the
extragradient SA scheme converges at the {\em optimal} rate of
$\Oscr\left(\frac{1}{{K}}\right)$, statements that were hitherto unavailable in
this regime.  Notably, we \uvs{optimize the initial steplength} by obtaining an
$\epsilon$-infimum of a discontinuous nonconvex function. Similar statements are
derived for mirror-prox generalizations and can accommodate monotone SVIs under
a weak-sharpness requirement. Finally, both the asymptotics and the empirical
rates of the schemes are studied on a set of variational problems where it is
seen that the theoretically specified initial steplength leads to significant
performance benefits.
\end{abstract}

\section{Introduction}
Several applications arising in engineering, science, finance, and
economics lead to a broad range of optimization and equilibrium
problems. Under suitable convexity assumptions, the equilibrium
conditions of such problems may
be compactly stated as as a variational inequality problem~\cite[Ch.~1]{facchinei02finite,konnov07equilibrium}. Recall that
given a set $X \subseteq \Real^n$ and a map $F:\Real^n \to \Real^n$,
the variational inequality problem, denoted by VI$(X,F)$, requires
an $x^* \in X$ such that
$ (x-x^*)^TF(x^*) \geq 0$ for all $x \, \in \, X.$
In a multitude of settings, either the evaluation of the map $F(x)$
is corrupted by error or its components are expectation valued, a
consequence of the original model being a stochastic optimization or
equilibrium problem.
Consequently, $F_i(x) \triangleq \mathbb{E}[F_i(x,\xi)]$
for $i = 1, \cdots, n$. Note that $\xi: \Omega \rightarrow
\mathbb{R}^{d}$, $F: X \times \mathbb{R}^{d} \rightarrow \mathbb{R}^{n}$
and $\left(\Omega,F,P \right)$ is the associated probability space.
As a result, the stochastic variational problem requires finding a vector
$x^* \in X$  such that
\begin{align}
(x-x^*)^T\bkE[F(x^*,\xi(\omega))]\geq 0, \hspace{3mm} \forall x \,\in \, X.
\label{eq:SVI}
\end{align}
Throughout, we use $F(x;\omega)$ to refer to $F(x,\xi(\omega))$ and refer to
our problem as SVI$(X,F)$. We begin by providing some motivation for weakening
the monotonicity requirement. 

\subsection{Motivation} We draw motivation from three classes of
problems\vs{.} 

\paragraph{\bf (a) Competitive exchange economy\vs{.}} {We consider a
competitive exchange economy~\cite{brighi2000} in which there is a collection of $n$ goods
with an associated price vector $p \in \Real_{++}^n$ and a positive
budget $w$. The consumption vector $F(p,w;\omega)$ specifies the
uncertain consumption level and the consumption has to satisfy budget
constraints in an expected-value sense, as specified by
$$ \mathbb{E}[p^T F(p,w;\omega)] \leq w. $$ This demand function is
assumed to be homogeneous with degree zero or $F(\lambda p,\lambda
		w;\omega) = F(p,w;\omega)$ for any positive $\lambda.$ \vs{Furthermore, $F(p,w) \triangleq \mathbb{E}[F(p,w;\omega)]$.} An
additional condition satisfied by $F(p,w;\omega)$ is the (expected) Weak
Weak Axiom of revealed preference (EWWA), which requires that for all
pairs $(p_1,w_1)$ and $(p_2,w_2)$
$$ p_2^T \mathbb{E}[ F(p_1,w_1;\omega)] \leq w_2 \implies
p_1^T\mathbb{E}[F(p_2,w_2;\omega)] \geq \vs{w_1}. $$
This axiom is an expected-value variant of WWA which itself represents a
weakening of the Weak Axiom of revealed preference~\cite{wa76}. Before
proceeding, we provide some intuition for this axiom. At prices $p_2$
and budget $w_2$, an agent chooses a bundle $F(p_2,w_2)$. If at the same
prices, she can also \vs{afford} $F(p_1,w_1)$, then we have that
$p_2^TF(p_1,w_1) \leq w_2$. Consequently, the consumer believes that the
bundle $F(p_2,w_2)$ is at least as good as $F(p_1,w_1)$. If at $p_1$ and
$w_1$, the bundle $F(p_2,w_2)$ is cheaper than the chosen bundle
$F(p_1,w_1)$, it follows that she can afford a bundle $b$ such that $b$
contains more of each commodity than $F(p_2,w_2)$. It may then be
concluded that the agent prefers $b$ to $F(p_2,w_2)$. But $F(p_2,w_2)$
is at least as good as $F(p_1,w_1)$, implying that the bundle $b$ is
preferable to $F(p_1,w_1)$. But \ka{bundle} $b$ is cheaper than $F(p_1,w_1)$,
		   contradicting the choice of $F(p_1,w_1)$ and one can conclude
		   that $F(p_2,w_2)$ cannot be cheaper than $F(p_1,w_1)$ or
		   $p_1^TF(p_2,w_2) \geq w_1.$}
{If $w_1 = w_2$, we have the following:
$$ p_2^T  F(p_1,w) \leq w \implies
p_1^TF(p_2,w) \geq w. $$
By the budget identity, $p_1^TF(p_1,w) = p_2^TF(p_2,w) = w$, implying that
$$ (p_2-p_1)^T  F(p_1,w) \leq 0 \implies
(p_2-p_1)^TF(p_2,w) \leq 0. $$
It follows that $F(\bullet,w)$ is a pseudomonotone map in $(\bullet)$
for any positive $w$. \ka{Formally, the property of pseudomonotonicity can be defined as follows.}
\begin{definition}[Pseudomonotonicity] \em
A continuous mapping  $F: X \subseteq \mathbb{R}^{n} \rightarrow \mathbb{R}^{n}$ is pseudomonotone on X if for all $x, y \in X$, 
$(x-y)^{T}F(y) \geq 0  \implies (x-y)^{T}F(x) \geq 0. $
\end{definition}

We now present how one may model the notion of
equilibrium in a general consumption sector with a finite set of
agents, denoted by $\Ascr$. An agent $a \in \Ascr$ is characterized by
an endowment $e_a \in \Real_{++}^n$ and a demand function $F_a(p,p^Te_a)$,
   implying that the consumption of agent $a$ is given by $\varphi_a(p)
	= p^T F_a(p,p^Te_a)$. The aggregate demand\vs{, given by the function
	$F(p)$, is} defined as $F(p) = \sum_{a \in \Ascr} F_a(p,p^Te_a).$ Note
	that $F(p)$ is homogeneous in $p$  with degree zero and by the
	individual budget identities, we have Walras' law; for all $p$,
	$p^TF(p) = p^T e = p^T (\sum_a e_a)$, where $e$ denotes the  sector-wide initial endowment. The demand function $F(p)$
	satisfies the WWA if
	$$ p_2\uvs{^T} F(p_1) \leq p_2^T e \implies p_1^T F(p_2) \geq p_1^Te.$$
The WWA can be presented in a more convenient fashion if $Z(p) = F(p) -
e$.} {While the consumption sector is captured by $Z(p)$, the set $Y$
represents the set of technology available and $y \in Y$ represents
either input (negative) or output (positive) based on sign. The set $Y$
satisfies two requirements: (i) \uvs{\em Free disposal} \uvs{:} goods may be
aribitrarily wasted without using further inputs or $-\Real_n^+ \subseteq
Y$; and (ii) \uvs{``\em No free lunch'':} production requires some inputs  or $Y
\cap \Real_+^n = \{0\}.$ Consequently, an equilibrium of the economy
$(Y,Z)$ is given by a $p \in {\cal P}$ such that
$$ (a) \, Z(p) \in Y \mbox{ and }  (b) \, p^T y \leq 0, \mbox{ for all } y \in Y. $$
Condition (a) implies that demand is met at equilibrium while (b)
	implies that firms maximize profits by choosing plans $y = Z(p)$. In
	fact, by ~\cite[Th.~1]{brighi2000}, $p$ is an equilibrium of the
	economy $(Y,Z)$ if and only if $p$ is a solution of VI$(Q,Z)$, where
	$Q \vs{ \ \triangleq \ } {\cal P} \cap Y^\circ$ and $Y^\circ \triangleq \{d: d^Ty \geq 0,
	y \in Y\}.$ But \uvs{by leveraging the WWA}, $Z(p) = \mathbb{E}[Z(p;\xi(\omega)]$ is a
			pseudomonotone map \uvs{from the EWWA}, leading to a pseudomonotone stochastic
			variational inequality problem. }

\paragraph{\bf (b) \bf Stochastic fractional programming\vs{.}} {\vs{Fractional programs} involve
the optimization of  metrics such as lift-to-drag
ratios in aircraft design~\cite{lift08}, \ka{fuel economy to engine
performance ratios} in automotive design~\cite{aym2005}, and signal-to-noise
ratios in wireless networks~\cite{signalnoise96}; these problems can often be cast as
pseudomonotone.  {More recently,} efforts in financial engineering optimize the Omega
ratio~\cite{omega02,omega03} which {quantifies the ratio of gain
probability to loss probability.}
All of the above problems fall under the umbrella of ``fractional
programming'' and {we consider the stochastic generalization} of this
problem:
\begin{align}
\min_{x \in X} \quad h(x)  \, {\triangleq} \,
\mathbb{E}\left[\frac{f(x;\xi(\omega))}{g(x;\xi(\omega))}\right],
\label{eq:fracconvinit}
\end{align}
where $f, g:\mathbb{R}^{n} \times \Real^d \rightarrow \mathbb{R}$ and
$\xi: \Omega \to \Real^d$. While $h(x)$ cannot be guaranteed to be
pseudoconvex in general, in automotive problems~\cite{aym2005},
$f(x;\xi(\omega))$ corresponds to the uncertain time taken to
accelerate from $0$ to $v^{\rm max}$ miles per hour while
$g(x;\xi(\omega))$ denotes the uncertain fuel economy.
The design space $x$ corresponds to engine design specifications
 such as   gear ratios and transmission switching levels.
Consequently, the equilibrium conditions are given by a pseudomonotone
	stochastic variational inequality problem. We present \vs{ 
	an extended version of Lemma 2.1 from~\cite{chandra72pseudo} as a definition} (proof omitted)\vs{. It} provides conditions
for \vs{the} pseudoconvexity of $h(x)$ under some basic assumptions. Note
that $h(x)$ is pseudoconvex if and only if $-h(x)$ is pseudoconcave.
\vvs{\begin{lemma}[{Stochastic pseudoconvex function}] \em
Suppose the following hold. (i) $f: X \times \Real^d \rightarrow \mathbb{R}$ is a nonnegative
convex \vs{function} in a.s. fashion; (ii) $g: X \times \Real^d \rightarrow
\mathbb{R}$ is a positive concave (strictly concave) function in an a.s. sense; and
(iii) $f(\bullet;\omega),g(\bullet;\omega) \in C^1$ in a.s. sense, then
the function $h: X  \rightarrow \mathbb{R}$, given by $h(x) = \mathbb{E}[f(x,\omega) \slash g(x,\omega)]$,
is a pseudoconvex (strongly pseudoconvex) function.
\label{def:pseudofrac}
\end{lemma}}
\paragraph{(c) \bf {Stochastic product} pricing\vs{.}} Consider an oligopolistic market with a set of substitutable goods in
which firms compete in prices. In such Bertrand markets~\cite{hobbs86mill,choi90}, the quantity sold by a particular firm is
contingent on the prices set by the firms (and possibly other product
		attributes) and this firm-specific demand
	is captured by the Generalized Extreme Value (GEV)
model~\cite{choi90,garrow2,netgev2}.
and the Multiplicative Competitive Interaction (MCI) model~\cite{MCI1,MCI3} have been very useful in characterizing consumer demand based on price and product attributes.
The \textit{multinomial logit} is a commonly used GEV
model that possesses some tractability and finds application in revenue
management problems in product pricing. We begin by defining the product
pricing problem for firm $j$:
\begin{align*}
{\max_{{p_j} \in \mathcal{P}_j} f_j({p})  \mbox{ where } f_j(p) \triangleq \mathbb{E}\left[
{p_j}\zeta_j({p;\omega})\right]},
\end{align*}
$p_j$ denotes the price
set by firm $j$ and $\zeta_j({p};\omega)$
{denotes} the demand {for product $j$}, defined \vs{as}
$${\zeta_j({p};\omega) \triangleq \frac{ e^{
-{\alpha_j(\omega) p_j}}}{c + \sum_{i=1}^N
	e^{-\alpha_i(\omega){p_i}}}},$$ where $\alpha_j(\omega)$
		\vs{denote} positive parameters for $j = 1, \hdots, N$.
{\vs{The resulting revenue function has been shown to be pseudoconcave
(see~\cite{gallego2013dynamic})}. The
relevance of this observation can be traced to the knowledge that under
a pseudoconvexity assumption, given $p_{-j}^*$,  $p_j^*$ is a stationary
point of this
problem if and only if  $p_j^*$ is a global minimizer of this problem.
Consequently, any solution to the collection of pseudomonotone variational inequality
problems is a  Nash-Bertrand equilibrium. Note that in Cournot
	or quantity games, suitably specified inverse demand functions also
lead to pseudoconcave revenue functions ~\cite[Th.~ 3.4]{econ11pseudo}.}

\subsection{Stochastic approximation schemes} The stochastic counterpart
of the variational inequality has received relatively less attention compared
to its deterministic counterpart.  Early efforts focused on the use of sample
average approximation (SAA) techniques and developed consistency statements  of
the resulting estimators~\cite{SPbook}. In fact, such techniques have been
applied towards the computation of solutions stochastic variational inequality
problems~\cite{xu10sample}. More recently, such avenues have been utilized to
develop confidence regions with suitable central limit
results~\cite{lu13confidence,lu14symmetric}.  An alternative approach  inspired
by the seminal work by Robbins and Monro~\cite{sa51robbins} is that of
stochastic approximation~\cite{Kush03,Borkar08,spall2005introduction}. \uvs{Via
averaging techniques~\cite{nemirovski78cezari,rupper88efficient,polyak90,polyak92}}, optimal rates in function values can be derived (also see
related work~\cite{kushner93stochastic,kushner95analysis} and prior
work~\cite{nemirovski83} on averaging). In the last decade, there has been a
surge of interest in the development of techniques for stochastic convex
optimization with a focus on optimal constant
steplengths~\cite{sa08nemirovski}, composite
problems~\cite{lan12optimal,lan13optimal} and
nonconvexity~\cite{bertsekas2000,lan13nonconvex}. However, in the context of
stochastic variational inequality problems, much of the prior work has been in
the context of monotone operators. Almost-sure convergence of the solution
iterates was first proven by Jiang and Xu~\cite{sajiang08} {under either strong/strict monotonicity or a variant of the acute angle condition}, while regularized
schemes for addressing merely monotone but Lipschitzian maps were subsequently
developed by Koshal et al.~\cite{sa13koshal}. In~\cite{wintersim13}, Yousefian
et al. weakened the Lipschitzian requirement by developing an SA scheme that
combined local smoothing and iterative regularization. From a rate standpoint,
there has been relatively less in the context of SVIs. A particularly
influential paper by Tauvel et al.~\cite{juditsky2011} proved that the
mirror-prox stochastic approximation scheme admits the optimal rate of
$\Oscr(1/\sqrt{K})$ in terms of the gap function when employing averaging over
monotone SVIs. In related work~\cite{yousefian14optimal}, Yousefian et al.
develop optimal extragradient-based robust smoothing schemes for monotone SVIs
in non-Lipshitzian regimes. {Noteworthy amongst past efforts is the
development of a class of accelerated techniques for deterministic and
stochastic variational inequality problems~\cite{chen17accelerated}.} We
summarize much of the prior results in Table~\ref{tab:difference}.  We believe
that this is amongst the first efforts to contend with this problem class but
it is worth noting that subsequent to the conference version of this
paper~\cite{kannan14pseudo}, there have been at least two papers that have
considered the solution of stochastic pseudomonotone variational inequality
problems. Of these, the first\footnote{Note that this paper references our
conference paper and a preprint of the current paper.} utilizes a similar
extragradient scheme with a.s. and rate statements that incorporates variable
batch size~\cite{iusem17extra}, \uvs{leading to an improved rate of
$\mathcal{O}(\tfrac{1}{K})$ in terms of $\mbox{dist}^2(\bar{x}_K,X^*)$}. In
addition, the second author \vs{of this paper} has also recently jointly coauthored work on a
block-coordinate variant of such schemes that incorporates a novel averaging
scheme in the context of stochastic mirror-prox schemes~\cite{yousefian17stochastic}.   
\begin{table}[htbp]
\newcommand{\tabincell}[2]{\begin{tabular}{@{}#1@{}}#2\end{tabular}}
{\scriptsize
\vspace{1mm}
\begin{center}
\begin{tabular}{|l|l|l|l|l|l|l|} \hline
Ref. & Applicability & SA Algorithm & Avg.  & Metric& Rate & a.s. \\ 
\hline \hline &  \tabincell{l}{\bf Strongly monotone} \\
\hline \hline
~\cite{sajiang08} & \tabincell{l}{Strongly monotone, \\  Lipschitz map} & {Proj. based} & N
& Iterates & - & Y \\ \hline
~\cite{yousefian13multiuser} & \tabincell{l}{Strongly monotone, \\ non-Lipschitz map} &
\tabincell{l}{Proj. based \\ + self-tuned step.} & N  & MSE (Soln. Iter.) & -  & Y \\ \hline
\hline
 &  \tabincell{l}{\bf Monotone + \\ \bf Single Proj.} \\ \hline \hline
~\cite{sajiang08} & \tabincell{l}{Monotone, \\ acute-angle condn.} & {Proj. based} & N
& Iterates & - & Y \\ \hline
~\cite{sa13koshal} & \tabincell{l}{Monotone, \\ Lipschitz map} &
\parbox{29mm}{Proj. based \\ + Regularization} & N  & Iterates & - & Y \\
\hline
~\cite{wintersim13} & \tabincell{l}{Monotone, \\ non-Lipschitz} &
\tabincell{l}{Proj. based \\ + Regularization\\
	 + smoothing} & N  & Iterates & - &
Y\\ \hline
~\cite{iusem17incremental} & \tabincell{l}{Monotone, \\
 Lipschitz \\
$X \triangleq \cap_{i=1}^m X_i$}
 &
\tabincell{l}{Proj-based \\
	+Regularization} & Y  & Gap fn. & ${\cal
	O}\left(\tfrac{\ln(k)}{\sqrt{k}}\right)$ & Y \\ \hline \hline
 &  \tabincell{l}{\bf Monotone + \\ \bf Extragradient-based}  \\\hline 
\hline 
~\cite{juditsky2011} & \tabincell{l}{Monotone,\\ $\|F(x)-F(y)\|_*$ \\
$\leq L\|x-y\|+B$} &
\parbox{29mm}{Mirror-prox} & Y  & Gap fn. & ${\cal
	O}\left(\tfrac{1}{\sqrt{k}}\right)$ &
N\\\hline
~\cite{yousefian14optimal} &\tabincell{l}{Monotone, \\
 Non-Lipschitz, \\
Bounded map} &
\tabincell{l}{Extragradient \\
+ Rand.~smoothing} & Y  & Gap fn. & ${\cal
	O}\left(\tfrac{1}{\sqrt{k}}\right)$ & Y \\ \hline
~\cite{yousefian17smoothing} & \tabincell{l}{Monotone \\ Non-Lipschitz} &
{Extragradient} & Y  & Gap fn. & ${\cal
	O}\left(\tfrac{1}{k^{1/6}}\right)$ & Y \\ \hline
~\cite{chen17accelerated} & \tabincell{l}{Monotone, \\ $F(x) = \nabla G(x)+H(x)$, \\
$G$ is $L_G$-smooth, \\ $H$ is $L_H$-Lipschitz} &
\tabincell{l}{Mirror-prox \\ + accelerated} & Y  & $-$ & $\mathcal{O}\left(\tfrac{L_G}{k^2}+\tfrac{L_H}{k}+\tfrac{c}{\sqrt{k}}\right)$ & N \\ \hline
\hline
 &  \tabincell{l}{\bf Pseudomonotone + \\ \bf Extragradient-based}  \\\hline \hline
~\cite{iusem17extra} & \tabincell{l}{Pseudomonotone \\ Lipschitz} &
\tabincell{l}{Extragradient\\+var. redn} & Y  & \tabincell{l}{$E[r(x_k)^2]$ \\
	(exp. squared-resid)} & ${\cal O}\left(\tfrac{1}{k}\right)$ & Y \\ \hline
\bf Sec. 3.3 & \bf \tabincell{l}{Strongly pseudo. or \\
	Strictly pseudo. or \\ 
	Pseudomonot.-plus} & \bf
\tabincell{l}{Extragradient,\\ Mirror-prox}& \bf N  & \bf  &\bf   & \bf Y \\  
 \hline
\bf  Sec. 3.4 & \bf \tabincell{l}{Strongly pseudo. or \\
	monotone+weak-sharp} & \bf
\tabincell{l}{Extragradient, \\ Mirror-prox}& \bf N  & \bf  MSE (Soln. Iter) & ${\cal
	O}\left(\tfrac{1}{{k}}\right)$ & \bf Y \\  
 \hline
\end{tabular}
\caption{SA based approaches for Stochastic Variational Inequality
	Problems (\vs{bold represents the current work})}
\label{tab:difference}
\end{center}
}
\vspace{-0.4in}
\end{table}

\subsection{Contributions and outline:} This paper  makes the following contributions:\\

\noindent {\bf (i) Almost-sure convergence \vs{.}} In
Section~\ref{sec:convergence}, we consider a {\em
stochastic} extragradient method and  show that the generated sequence of iterates
converges to a solution in an almost-sure sense. We refine these
statements to a subclass of
non-monotone problems and extend the
convergence statement to the mirror-prox regime. To the best of our
knowledge, there appears to be no \vs{prior}\ a.s. convergence theory for
this class of SVIs.\\

\noindent{\bf (ii) Rate analysis \vs{.}} Under slightly stronger settings of
pseudomonotonicty, in Section~\ref{sec:rate}, we prove that the extragradient
scheme displays the optimal rate for strongly pseudomonotone maps.
Additionally, a similar statement is proved for the mirror-prox generalization
as well as for problems characterized by the weak-sharpness property.  {In
particular, we emphasize that our work derives rate estimates for the {\em
iterates} without resorting to averaging, in contrast with available statements
for monotone SVIs that provide rate statements in terms of the gap function.}
Notably, in all three cases, we further refine the bound by selecting a suitable initial
steplength by  deriving an ($\epsilon$-) infimum of a nonconvex
discontinuous function in terms of problem-specific constants (the strong
pseudomonotonicity constant, Lipschitz constant, compactness measures, etc.)}
Again, this appears to be amongst the first rate statements in the regime of
pseudomonotone problems. \\

\noindent{\bf (iii) Numerical Results:} In Section~\ref{sec:numbers},
	based on a test suite of problems, we
	examine the empirical behavior of our schemes and note the benefits
	seen from optimizing the initial steplength.

\section{Extragradient-based stochastic approximation schemes}\label{sec:convergence}
\subsection{Background and Assumptions}
Given an $x_0 \in X$ in a
	traditional SA scheme and a steplength sequence $\{\gamma_k\}$, a
sequence $\{x_k\}$ is constructed via the following update
rule:
\begin{align}\tag{SA}
x_{k+1} & := \Pi_X(x_k -\gamma_k (F(x_k)+w_k)), \quad k \geq 0,
\label{eq:SSP}
\end{align}
where $w_k$ is defined as $w_k \ {\triangleq} \ F(x;\omega_k)-\uvs{F(x_k)}.$
We consider a
{\em stochastic} extragradient scheme akin to that presented
in~\cite{juditsky2011}. Given an $x_0 \in
X$ and a steplength sequence $\{\gamma_k\}$, this scheme comprises of two steps for $k \geq 0$:
\begin{align}\tag{ESA}
\begin{aligned}
x_{k+1/2} & := \Pi_X(x_k -\gamma_k (F(x_k)+w_k)),\\
x_{k+1}   & := \Pi_X(x_k -\gamma_k (F(x_{k+1/2})+w_{k+1/2})),
\end{aligned}
\label{eq:ESA}
\end{align}
{where $w_{k+1/2} \triangleq F(x_{k+1/2};\omega_{k+1/2}) -
F(x_{k+1/2}).$ }
At any iteration $k$, the history $\Fscr_{k}$  and $\Fscr_{k+1/2}$
	are defined as 
{$\mathcal{F}_k \triangleq \sigma \left\{x^0, \omega_0, \omega_{1/2}, \omega_1,
	\hdots, {\omega_{k-1/2}} \right\}$ and $\mathcal{F}_{k+1/2}
\triangleq \mathcal{F}_k \cup \{\omega_{k}\}$, respectively.} 
Next, we define the property of pseudomonotonicity and its variants and
denote the solution set of SVI$(X,F)$ by $X^*$:
\begin{definition}[Pseudomonotonicity and variants] \em
Consider a continuous mapping $F: X \subseteq \mathbb{R}^{n} \rightarrow \mathbb{R}^{n}$. Then, the following hold:
\begin{enumerate}
\item[(i)] $F$ is pseudomonotone on $X$ if for all $x, y \in X$,
$ (x-y)^{T}F(y) \geq 0  \implies (x-y)^{T}F(x) \geq 0.$
\item[(ii)] $F$ is pseudomonotone-plus on $X$ if it is pseudo-monotone on $X$ and for all vectors $x$ and $y$ in $X$,
$(x-y)^TF(y)\geq 0,  \hspace{1mm} (x-y)^TF(x) = 0  \implies F(x) = F(y).$
\item[(iii)] $F$ is strictly pseudomonotone on $X$, if for all $x, y \in X$,
$(x-y)^{T}F(y) \geq 0  \implies (x-y)^{T}F(x) > 0$ where {$x \neq y$.}
\item[(iv)] $F$ is strongly pseudomonotone on $X$ if for all $x, y \in
X$, there exists $\sigma > 0$ such that
$(x-y)^{T}F(y) \geq 0  \implies (x-y)^{T}F(x) \geq \sigma \|x-y\|^2,$
{where $x \neq y$.}
\item [(v)] The acute angle relation holds if for {any} $x \in X \backslash {X^*}$ and $x^* \in X^*$ ,
\begin{align}
(x-x^*)^{T}F(x) > 0.
\label{eq:spseudo}
\end{align}
\item[(vi)] The weak sharpness property holds
if there exists an $\alpha>0$ such that
$$ (x-x^{*})^TF(x^{*}) \geq \alpha \min_{x^* \in X^*}\|x-x^{*}\|,
	\hspace{5mm} \forall x \in X, \quad \forall x^* \in X^*.$$
\end{enumerate}
\label{def:pseudomonotone}
\end{definition}
\uvs{It is worth recalling that strict pseudomonotonicity implies the acute angle condition.}
Next, we make the following assumptions on the conditional first and second
moments:
\begin{assumption}[{\bf Unbiasedness and boundedness of conditional
	second moments}] \em
At an iteration $k$, the following hold in an a.s. sense:
	\be
		\item[(A1)] The \vs{conditional first moments} $\bkE[w_{k} \mid \mathcal{F}_{k}]$ and
$\bkE[w_{k+1/2}\mid \mathcal{F}_{k+1/2}]$ are zero;
\item[(A2)] The conditional second moments are bounded  \uvs{a.s. in that there exists a $\nu > 0$ such that }  $\bkE[\| w_{k} \|^2 \mid
\mathcal F_{k}] \leq \nu^2$ and  $\bkE[\| w_{k+1/2} \|^2 \mid
\mathcal F_{k+1/2}] \leq \nu^2$ \uvs{ for all $k$}.
\ee
\end{assumption}
We now provide  assumptions on steplength sequences
consistent with most SA schemes.
\begin{assumption}[{\bf Square summability and non-summability of steplength sequences}]\em
The diminishing sequence $\{\gamma_k\}$  satisfies the following:
\be
\item[(A3)] The steplength sequence is square-summable\uvs{:} ${\sum_{k=0}^{\infty}} \gamma_{k}^{2} < \infty.$
\item[(A4)] The steplength sequence is non-summable\uvs{:}  ${\sum_{k=0}^{\infty}} \gamma_{k} = \infty.$
\ee
\end{assumption}

We impose a further requirement on the map given by the following:
{\begin{assumption}[{\bf Lipschitz continuity and boundedness of $F$}]\em
(A5)  $F(x)$ is Lipschitz continuous and
	bounded over $X$ \ie there exist positive scalars $L$ and $B$
	such that for all $x,y \in X$
$\|F(x)-F(y) \|\leq L\|x-y\|$ and $\| F(x) \| \leq \frac{B}{2}.$ 
\end{assumption}}
We use the following super-martingale convergence results~\cite[Lemma~10,11,
   page~49,50]{polyak87introduction}.
{\begin{lemma}
Let $V_k$ be a sequence of
nonnegative random variables adapted to $\sigma$-algebra $\mathcal{F}_k$ and
such that
$\mathbb{E}[V_{k+1} \mid \mathcal{F}_{k}] \leq  (1-\delta_k)V_k + \psi_k,
	\hspace{2mm} \forall k \geq 0, \mbox{ a.s. } $
where $0 \leq \delta_k \leq 1, \psi_k \geq  0$, and
$\sum_{k=0}^{\infty} \delta_k = \infty,
\sum_{k=0}^{\infty} \psi_k < \infty, and \lim_{k \rightarrow \infty} \frac{\psi_{k}}{\delta_{k}} = 0.$
Then,  $V_k \rightarrow 0$ in a.s. sense.
\label{lemm:expseq1}
\end{lemma}}
{\begin{lemma}
Let $V_k, u_k, \psi_k$ and $\gamma_k$ be nonnegative random variables adapted to
$\sigma$ algebra $\mathcal{F}_k$. If a.s
$\sum_{k=0}^{\infty} \delta_k < \infty, \sum_{k=0}^{\infty}\psi_k < \infty$, and
$\mathbb{E}[V_{k+1} \mid \mathcal{F}_k] \leq (1+\delta_k)V_k - u_k + \psi_k, \hspace{2mm} \forall k \geq 0,$
then  {$V_k$} is convergent in an a.s. sense and
$\sum_{k=0}^{\infty} u_k < \infty$ in an a.s. sense.
\label{lemm:expseq2}
\end{lemma}}
\vspace{-0.2in}
\subsection{An extragradient SA scheme}
In this subsection, we prove that the iterates generated
by the \eqref{eq:ESA} scheme converge to the solution set of the
original problem in an almost sure sense by leveraging ideas drawn from
the proof of the deterministic version presented in~\cite[Lemma~12.1.10]{facchinei02finite}.
A challenge in this scheme arises due to the  two independent stochastic
errors \ka{respectively from the two sub-steps} at every iteration and the lack of a
direct expression of $x_{k+1}$ in terms of $F(x_k)$ unlike the standard
projection scheme.  We begin by relating any two successive iterates in
Lemma~\ref{lemm:descexg}.
\begin{lemma} \em
\label{lemm:descexg}
Consider the stochastic variational inequality problem defined by
SVI$(X,F)$ and let $x^*$ denote any solution of SVI$(X,F)$. Suppose
Assumption (A5) holds and consider the
sequence of the iterates be generated by the extragradient scheme~\eqref{eq:ESA}
and let {$u_k \triangleq 2\gamma_k F(x_{k})^T(x_{k}-x^{*})$.}
Then, the following {holds} for any iterate $k$:
\begin{align}
\|x_{k+1} - x^*\|^2 & \leq \uvs{\left(1+\frac{\gamma_k^2}{\beta}\right)} \|x_k - x^*\|^2  - u_k -2\gamma_k w_{k+1/2}^T (x_k-x^{*}) + \gamma_k^2 t_k,
\label{eq:iteraterelation}
\end{align}
where the scalar $t_k \geq 0$ is an appropriately defined scalar.
\end{lemma}
\begin{proof}
Let $y_k = {x_{k}} - \gamma_k (F(x_{k+1/2})+w_{k+1/2})$. Then,
\begin{align*}
\|x_{k+1}-x^*\|^2  = \|\Pi_X(y_k) - x^*\|^2
 = \|y_k-x^*\|^2 + \|\Pi_X(y_k)-y_k\|^2 +
2\left(\Pi_X(y_k)-y_k\right)^T(y_k -
x^*).
\end{align*}
\begin{align*}
\mbox{ Note that } & \quad 2\|y_k - \Pi_X(y_k)\|^2 +  2 (\Pi_X(y_k)-y_k)^T (y_k-x^*) \\
& =  2\|y_k - \Pi_X(y_k)\|^2 +  2 (\Pi_X(y_k)-y_k)^T
(y_k-\Pi_X(y_k)+\Pi_X(y_k) - x^*) \\
& =   2\|y_k - \Pi_X(y_k)\|^2 - 2\|y_k - \Pi_X(y_k)\|^2 + 2 (\Pi_X(y_k)-y_k)^T
(\Pi_X(y_k) - x^*) \\
& =  2 (\Pi_X(y_k)-y_k)^T
(\Pi_X(y_k) - x^*) \leq 0,
\end{align*}
where the last inequality follows from the projection property. As a
consequence, we have that
\begin{align}
\|y_k - \Pi_X(y_k)\|^2 +  2 (\Pi_X(y_k)-y_k)^T (y_k-x^*)  \leq
- \|y_k - \Pi_X(y_k)\|^2.
\label{eq:firsteqprop}
\end{align}
By invoking ~(\ref{eq:firsteqprop}) in the expansion of
$\|x_{k+1}-x^*\|^2$, we obtain
\begin{align*}
	\|x_{k+1}-x^*\|^2 & = \|y_k-x^*\|^2 + \|\Pi_X(y_k)-y_k\|^2 + 2(\Pi_X(y_k)-x^*)^T(y_k -
			x^*) \\
	& {\leq \|y_k-x^*\|^2 - \|y_k - \Pi_X(y_k) \|^2} \\
	& = \|x_{k} - \gamma_k (F(x_{k+1/2})+w_{k+1/2})-x^*\|^2 -
		\|x_{k} -
	\gamma_k (F(x_{k+1/2})+w_{k+1/2}) - x_{k+1} \|^2 \\
	& = \|x_{k} - x^*\|^2 + \gamma_k^2
	\|F(x_{k+1/2})+w_{k+1/2})\|^2 -
	2\gamma_k(x_{k}-x^*)^T(F(x_{k+1/2})+w_{k+1/2}) \\
 & - \|x_{k+1} - x_{k}\|^2 - \gamma_k^2
	\|F(x_{k+1/2})+w_{k+1/2})\|^2 +
	2\gamma_k(x_{k}-x_{k+1})^T(F(x_{k+1/2})+w_{k+1/2}) \\
	& = {\|x_{k} - x^*\|^2 - \|x_{k}-x_{k+1}\|^2 + 2\gamma_k(x^*-x_{k+1})^T(
	F(x_{k+1/2})+w_{k+1/2})}.
\end{align*}
By adding and subtracting $x_k^T(F(x_{k+1/2}\us{)}+w_{k+1/2})$, we obtain
 \begin{align*}
	\|x_{k+1}-x^*\|^2 & \leq \|x_{k} - x^*\|^2 - \|x_{k}-x_{k+1}\|^2 + 2\gamma_k(x^*-x_{k+1})^T(
	F(x_{k+1/2})+w_{k+1/2}) \\
			& =  \|x_{k} - x^*\|^2 - \|x_{k}-x_{k+1}\|^2 + 2\gamma_k(x^*-x_{k})^T(
	F(x_{k+1/2})+w_{k+1/2}) \\
			& + 2\gamma_k(x_k-x_{k+1})^T(
	F(x_{k+1/2})+w_{k+1/2}) \\
			&  \leq  \|x_{k} - x^*\|^2 - \|x_k -x_{k+1}\|^2 + 2\gamma_k(x^*-x_{k})^T(
	F(x_{k+1/2})+w_{k+1/2}) \\
			& + \|x_k - x_{k+1}\|^2 + \gamma^2_k\|
	F(x_{k+1/2})+w_{k+1/2} \|^2 \\
			&  =  \|x_{k} - x^*\|^2  + 2\gamma_k(x^*-x_{k})^TF(x_k) + \underbrace{2\gamma_k (x^*-x_{k})^T(
	F(x_{k+1/2}\uvs{)} - F(x_k))}_{\tiny \mbox{Term a}} \\
			& + 2\gamma_k(x^*-x_{k})^Tw_{k+1/2}+ \gamma^2_k\|
	F(x_{k+1/2})+w_{k+1/2} \|^2 .
\end{align*}
Next, we observe that {Term a can be bounded as follows:}
\begin{align*}
\mbox{Term a} & \leq \frac{\gamma_k^2}{\beta} \|x^*-x_k \|^2 + \beta \|F(x_{k+1/2}-F(x_{k})) \|^2  \leq \frac{\gamma_k^2}{\beta} \|x^*-x_k \|^2 + \beta L^2 \|x_{k+1/2}-x_{k} \|^2 \\
& \leq \frac{\gamma_k^2}{\beta} \|x^*-x_k \|^2 + \beta L^2 \|\Pi_{X} (x_{k}-\gamma_k (F(x_{k})+w_{k}))-\Pi_{X}(x_{k}) \|^2 \\
& = \frac{\gamma_k^2}{\beta} \|x^*-x_k \|^2 + \gamma_k^2 \beta L^2 \|F(x_{k})+w_{k}\|^2.
\end{align*}
As a consequence, we have that
\begin{align}
\|x_{k+1}-x^*\|^2  & \leq \left( 1+\frac{\gamma_k^2}{\beta}\right) \|x_{k} - x^*\|^2 - 2\gamma_k(x_k-x^*)^TF(x_k)
 + 2\gamma_k(x^*-x_{k})^Tw_{k+1/2} \notag\\
& + \gamma^2_k \left( \|F(x_{k+1/2}) \|^2 + \| w_{k+1/2} \|^2 + \beta
	L^2 \| F(x_k) \|^2 + \beta L^2 \| w_k \|^2 \right) \notag\\
& + 2 \gamma_k^2 \left(\uvs{\beta L^2}w_k^TF(x_k) + w_{k+1/2}^T
		F(x_{k+1/2}) \right) \notag \\
& \leq \left( 1+\frac{\gamma_k^2}{\beta}\right) \|x_{k} - x^*\|^2 - \overbrace{2\gamma_k(x_k-x^*)^TF(x_k)}^{ =  u_k}
\label{lip-bound-rec}
 + 2\gamma_k(x^*-x_{k})^Tw_{k+1/2}  \\
& + \gamma^2_k \overbrace{\left( \frac{B^2(1+\beta L^2)}{4} + \| w_{k+1/2} \|^2  + \beta L^2 \| w_k \|^2 + 2w_{k+1/2}^T F(x_{k+1/2})
+ 2\beta L^2 w_k^TF(x_k)  \right)}^{ =
t_k},\notag
\end{align}
\ka{where the last expression follows by invoking the boundedness of $F$ over $X$.}
\end{proof}
{While the above Lemma relates iterates
for general mappings $F$, we 
begin with an analysis on pseudomonotone-plus problems.}
\begin{proposition}[{\bf a.s. convergence of ESA}]\label{prop:pseudo}
\em {Consider SVI$(X,F)$  \uvs{where $F$ is assumed to be a pseudomonotone-plus mapping.} Suppose assumptions (A1)--(A5) hold.}
Then, the extragradient scheme~(\ref{eq:ESA}) generates a sequence $\{x_k\}$ such that \uvs{$\{x_k\}$ is bounded a.s. and any limit point of $\{x_k\}$ is a solution of SVI$(X,F)$ 
in an a.s. sense.} 
\end{proposition}
\begin{proof}	
Consider the recursion \eqref{lip-bound-rec}. Taking expectations
conditioned on ${\cal F}_k$, we have that

{
\begin{align}
 \mathbb{E}[\|x_{k+1}-x^*\|^2\mid {\mathcal F}_{k}] & \leq 	\left( 1+\frac{\gamma_k^2}{\beta}\right)\|x_k - x^*\|^2 - u_k -2\gamma_k \bkE[(x_k-x^*)^T w_{k+1/2} \mid {\mathcal F}_k] \nonumber \\
& + \gamma_k^2 \left( \frac{B^2(1+\beta L^2)}{4}+\bkE[\| w_{k+1/2}\|^2 \mid \mathcal{F}_k] + \beta L^2 \bkE[\|w_{k} \|^2 \mid \mathcal{F}_k] \right) \nonumber \\
& +  \gamma_k^2 \left( \bkE[2w_{k+1/2}^{T}F(x_{k+1/2}) \mid
		\mathcal{F}_k] + 2\beta L^2 \bkE[w_{k}^TF(x_{k})|\mathcal{F}_{k}]  \right) \nonumber\\
& = \left( 1+\frac{\gamma_k^2}{\beta}\right)\|x_k - x^*\|^2 - u_k -2\gamma_k \bkE[\bkE[(x_k-x^*)^T w_{k+1/2} \mid {\mathcal F}_{k+1/2}]\uvs{\mid {\mathcal F}_{k}}] \nonumber \\
& + \gamma_k^2 \left( \frac{B^2(1+\beta L^2)}{4}+\bkE[\bkE[\| w_{k+1/2}\|^2 \mid \mathcal{F}_{k+1/2}]\mid \mathcal{F}_k] + \beta L^2 \bkE[\|w_{k} \|^2 \mid \mathcal{F}_k] \right) \nonumber \\
& +  \gamma_k^2 \left( \bkE[\bkE[2w_{k+1/2}^{T}F(x_{k+1/2}) \mid
		\mathcal{F}_{k+1/2}]\mid \mathcal{F}_k] + 2\beta L^2
		\bkE[w_{k}^TF(x_{k})\mid\mathcal{F}_{k}]  \right) \nonumber \\
\label{pseudo-ineq}
& \leq \left( 1+{\delta_k}\right)\|x_k - x^*\|^2 - u_k + {\psi_k}, \\
\mbox{where} \quad
\delta_k &\triangleq\frac{\gamma_k^2}{\beta} \mbox{ and } \psi_k
	\triangleq \gamma_k^2 \left(\frac{(B^2+8\nu^2)(1+\beta
				L^2)}{4}\right), \notag
\end{align}
and the inequalities follow from using the tower law and assumption (A1).
The remainder of the proof requires application of the super-martingale
convergence theorem (Lemma~\ref{lemm:expseq2}). \ka{This requires that $u_k \geq 0$ for all $k$, a fact that follows from noting that $F$ is a pseudomonotone map over $X$, implying that 
$$ (x_k-x^*)^TF(x^*) \geq 0 \implies \overbrace{(x_k-x^*)^TF(x_k)}^{ \triangleq \tfrac{u_k}{\gamma_k}} \geq 0. $$}
From the solution property of $F(x^*)^T(x_{k}-x^*) \geq 0$ and from the pseudomonotonicity of $F(x)$,
$u_k = 2\gamma_k F(x_{k})^T(x_{k}-x^*) \geq 0$.
Using assumption (A3), it is observed that ${\sum_k \psi_k} < \infty$.  Invoking Lemma~\ref{lemm:expseq2}, we have that $\{\|x_k -
x^*\|^2\}$ is a convergent sequence in an a.s. sense. 
\vvs{Then in an a.s. sense, it follows that $\{x_k\}_{k \geq 0}$ is a bounded sequence and has a convergent subsequence $\{x_k\}_{k \in \mathcal{K}}$.} 
\vvs{We proceed by contradiction;} suppose $x_{k}$ converges} to $\hat{x}$ \vvs{along subsequence $\Kscr$} where $\hat x$ is not necessarily a solution to
SVI$(X,F)$.
{Since $\{\delta_k\}$ is summable in an a.s. sense,
$\{u_k\}$ is summable a.s. and from the
non-summability of $\gamma_k$,  we have that a.s., the following implication holds.
\begin{align}
\sum_{\vvs{k \in \Kscr}} u_k & = \sum_{\vvs{k \in \Kscr}} 2\gamma_k F(x_{k})^{T}(x_{k}-x^{*})  <
\infty \hspace{1mm} \nonumber\implies
{\lim_{\vvs{k \in \Kscr}}} \ F(x_{k})^{T}(x_{k}-x^{*}) = 0.
\end{align} 
 Since $x_{k} \xrightarrow[k \to \infty]{a.s.} \hat{x}$ along the subsequence $\Kscr$ from \eqref{pseudo-ineq} and by the continuity of $F$ over $X$,
we obtain
\begin{align}
F(\hat{x})^{T}(\hat{x}-x^*) = 0.
\label{eq:zerosF}
\end{align}}
\uvs{By recalling that $x^*$ is a solution of VI$(X,F)$ and by invoking} the pseudomonotone-plus property of $F$ together with  (\ref{eq:zerosF}), we have
\begin{align}
\left[ F(x^*)^{T}(\hat{x}-x^{*}) \geq 0 \text{and} F(\hat{x})^{T}(\hat{x}-x^*) = 0 \right] \Longrightarrow 	F(\hat{x}) = F(x^*).
\label{eq:eqF}
\end{align}
Therefore from~(\ref{eq:eqF}),  the following holds:
\begin{align*}
\forall x \in X, \hspace{3mm} F(\hat{x})^{T}(x-\hat{x}) & = F(x^*)^{T}(x-\hat{x})
 = F(x^*)^{T}(x-x^*)+F(x^*)^{T}(x^*-\hat{x}) \\
& \geq F(x^*)^{T}(x^*-\hat{x})
 = F(\hat{x})^{T}(x^*-\hat{x}) = 0,
\end{align*}
where the last equality follows from~(\ref{eq:zerosF}). It follows that
	$\hat{x}$ is a solution to SVI$(X,F)$ \uvs{and any limit point of $\{x_k\}$  is a solution of SVI$(X,F)$ in an a.s. sense}.
\end{proof}
Next, we
extend the convergence theory to accommodate variants of
pseudomonotonicity as well as problems satisfing the acute
angle property.

\begin{proposition} [{\bf a.s. convergence of ESA under weaker
	conditions}] \em
Consider $\textrm{SVI}(X,F)$ and let {assumptions (A1-A5) hold.}
Consider one of the following statements:
\begin{enumerate}
\item[(i)] $F$ satisfies the acute angle relation at $X^*$ given by~\eqref{eq:spseudo}.
\item[(ii)]   $F$ is strictly pseudomonotone on $X$.
\item[(iii)]  $F$ is strongly pseudomonotone on $X$.
\item[(iv)] $F$ is pseudomonotone on $X$  and  is given by the gradient of
$\mathbb{E}[f(x,\omega)]$.
\end{enumerate}
Then, the extragradient
scheme~(\ref{eq:ESA}) generates a sequence $\{x_k\}$
\uvs{such that $\{x_k\}$ is bounded a.s. and any limit point of $\{x_k\}$ is a solution of SVI$(X,F)$ 
in an a.s. sense.} 
\label{prop:pseudoadd}
\end{proposition}
\begin{proof} We begin from \eqref{pseudo-ineq} in the proof of
Prop.~\ref{prop:pseudo} and instead of the pseudomonotone-plus property,
	we impose properties imposed by (i) -- (iv).  \\

(i) From~(\ref{eq:spseudo}),
$u_k = 2\gamma_kF(x_{k})^{T}(x_{k}-x^*) > 0$ and 
$\delta_k$ and $\psi_k$ are summable since $\gamma_k$ is a square summable
sequence. Invoking Lemma~\ref{lemm:expseq2}, we have that $\sum_{k} u_k <
\infty$ a.s. and \vvs{$\{\|x_k-x^*\|^2\}$ is a convergent sequence in an a.s.
sense, implying that $\{x_k\}_{k \geq 0}$ is a bounded sequence in an a.s.
sense}. \vvs{It follows that $\{x_k\}$ has a convergent subsequence indexed by
$\Kscr$ with limit point $\hat x$. We proceed by contradiction and assume that $\hat{x} \not \in X^*$.} {Recall that $\sum_{k} \gamma_k = \infty$ \vvs{and
$\sum_{k \in \Kscr} u_k < \infty$}, implying that \vvs{$\lim_{k \rightarrow
\infty}  F(x_{k})^{T}(x_{k}-x^*)= 0$} in an a.s. sense. \vvs{Consequently, by
continuity of $F$ and by recalling that $x_k \xrightarrow[a.s.]{k \to \infty} \hat{x}$ along $\Kscr$, we
have that $F(\hat{x})^T(\hat{x}-x^*) = 0$.} \vvs{But this contradicts} the acute angle property and it follows that $\hat x \in
	X^*$.  }\\

\noindent (ii) Since $F$ is strictly pseudomonotone, $F$ satisfies the acute angle
relation  and the result follows.\\

\noindent (iii) Since $F$ is strongly pseudomonotone, $F$ is strictly
pseudomonotone and the result follows. \\

\noindent (iv) From the first part of the  proof of
Prop.~\ref{prop:pseudo}, since the map is pseudomonotone we have that
	\vvs{$\|\{x_{k}-x^*\|^2\}$ is a
	convergent sequence in an a.s. sense, implying that in an a.s. sense, $\{x_k\}_{k \geq 0}$ is a bounded sequence  and has a convergent subsequence with index set $\Kscr$ and limit point $\hat{x}$}. We proceed by
		contradiction and assume that $\hat x \not \in X^*$. {By the pseudoconvexity {of}
	$f(x)$, {for any $x_1, x_2 \in X$,
$ \nabla f(x_1)^T(x_2-x_1)  \geq   0 \hspace{3mm} \implies f(x_2)
	\geq f(x_1), $ implying that}
\begin{align}
\nabla f(x^*)^T(\hat{x} - x^*)   \geq    0 \implies f(\hat x) \geq
	f(x^*).
\label{eq:globalpseudo}
\end{align}
{But from \eqref{eq:zerosF}, we have that $\nabla f(\hat{x})^T(x^*-\hat{x})  =
0$ implying that  $f(\hat x) \leq f(x^*)$. It follows that  
 $f(\hat{x}) = f(x^*)$ and $\hat{x}$ is a global minimizer of
 $\mathbb{E}[f(x,\xi)]$ over $X$ and a solution of
SVI$(X,F)$.}}
\end{proof}
{Next, we consider a monotone regime where
	VI$(X,F)$ satisfies a weak sharpness property.  Prior to providing
		our main convergence statement, we provide an
		intermediate lemma. }
\begin{proposition}[a.s. convergence under weak-sharpness and
monotonicity]
\label{prop:hypoasure}
\em Let (A1-A5) hold. Consider SVI$(X,F)$ where $F$ is a
continuous monotone map over $X$.
Suppose the weak sharpness property holds for the mapping $F$ and the
solution set $X^*$ {with parameter $\alpha$}.
Then, the extragradient
scheme~(\ref{eq:ESA}) generates a sequence $\{x_k\}$
\uvs{such that $\{x_k\}$ \vvs{converges to a solution of SVI$(X,F)$ in an a.s. sense.}} 

\end{proposition}
\begin{proof}
{We begin by restating~(\ref{pseudo-ineq}) as follows:
\begin{align}
 \mathbb{E}[\|x_{k+1}-x^*\|^2\mid {\mathcal F}_{k}] & \leq  \left( 1+\frac{\gamma_k^2}{\beta}\right)\|x_k - x^*\|^2 - u_k + \gamma_k^2 \left(\frac{(B^2+8\nu^2)(1+\beta L^2)}{4}\right).
\end{align}
where $-u_k = -2\gamma_k F(x_{k})^{T}(x_{k}-x^*) \leq - 2\gamma_k \alpha
	\mbox{dist}(x_k,X^*)$ by  the weak-sharpness property. 
This further implies that
\begin{align}
 \mathbb{E}[\|x_{k+1}-x^*\|^2 \mid \Fscr_k]
 & \leq  \left( 1+\frac{\gamma_k^2}{\beta}\right)\|x_k - x^*\|^2
	-\gamma_k {\alpha} \mbox{dist}(x_{k},X^*) \notag \\ & +
\gamma_k^2 \left(\frac{(B^2+8\nu^2)(1+\beta L^2)}{4}\right).
\label{weaksharp_descent}
 \end{align}
From the super-martingale convergence Lemma, it can be seen that  the following
hold in an a.s. sense: (i) $\{\|x_k-x^*\|^2\}$ is a convergent sequence; and
(ii) $\sum_k \gamma_k  \alpha \mbox{dist}(x_{k},X^*) < \infty.$ \vvs{We proceed
to show that $\mbox{dist}(x_k,X^*) \xrightarrow[a.s.]{k \to \infty} 0$; equivalently, this implies that almost every subsequence $\{x_k\}_{k \in \mathcal{K}} \to X^*$ as $k \to \infty$.  Suppose this is
false.  Then \vs{for $\omega \in \Omega_1 \subset \Omega$ and $\mu(\Omega_1) > 0$ (i.e. finite probability), $\{x_k\}_{k \in \Kscr(\omega)} \ \to
\ \widehat{x}(\omega)$, $\mbox{dist}(x_k,X^*) \xrightarrow[a.s.]{k \to \infty}
v(\omega) > 0$ where  $\Kscr(\omega)$
denotes a random subsequence with limit point $\widehat{x}(\omega)$},
$v(\omega)$ is a random positive scalar \vs{ and}
$\inf_{\omega \in \Omega_1} v(\omega) \geq \bar{v}$. \vs{Then for every $\omega
\in \Omega_1$ and for $\bar{v} > 0$, there exists a $K(\omega)$ such that
$\|x_k - \widehat{x}(\omega)\| \leq \bar{v}/2$ for $k \in \Kscr(\omega)$ and $k
\geq K(\omega)$.} It follows that that for $\omega \in \Omega_1$ (with finite
probability), we have that $$\sum_{k \in \Kscr(\omega)} \gamma_k
\mbox{dist}(x_k,X^*)  \geq \sum_{k \geq K(\omega), k \in \Kscr(\omega)}
\gamma_k \mbox{dist}(x_k,X^*) =  \infty.$$  But this contradicts that claim
that
$\sum_{k=1}^{\infty} \gamma_k \mbox{dist}(x_k,X^*) < \infty$ almost surely.
Consequently, $\mbox{dist}(x_k,X^*) \xrightarrow[a.s.]{k \to \infty} 0$.}} 
\end{proof}
\noindent {\bf Remark}: {A natural question emerges as to the relevance
of this result, given that monotone maps have been studied
in the past (cf.~\cite{juditsky2011,sa13koshal}). First, we present statements that show that the
original sequence converges almost surely to a point in the solution
set, rather than showing the averaged counterpart coverges to the
solution set. Second, in contrast with
~\cite{sa13koshal}, we do not resort to regularization in
deriving almost-sure convergence statements. Third, we proceed to show
that under the prescribed assumptions, we obtain the optimal rate of
convergence in the solution estimators, rather than the gap function.
Finally, we do not require that $X$ be bounded for claiming the a.s.
convergence of solution iterates in this regime.}

%
	


\subsection{Mirror-prox generalizations}
Given a point and a set, the Euclidean projection computes the
closest point in this set by using the Euclidean norm as the distance
metric. A generalization to this
operation~\cite{nemirovski04prox,sa08nemirovski} utilizes a class of
distance functions that include the Euclidean norm as a special case.
Given a distance metric $s(x)$, the prox function $V(x,z)$ takes the
form:
\begin{align}
V(x,z) \triangleq s(z) - \left[s(x)+ \nabla s(x)^{T}(z-x) \right],
\label{eq:proxfn}
\end{align}
where $s(x)$ is assumed to be a strongly convex and differentiable function in $x$. The
resulting prox  subproblem is given by the following:
\begin{align}
\P_{X}(x,r) \triangleq \mbox{arg}\min_{z \in X} \left( r^{T}z + V(x,z)
		\right).
\label{eq:proxproj}
\end{align}
It may be observed  that when ${s(x)} = \half \| x\|^2$, then
	$V(x,z) = \half \|x-z\|^2$. Furthermore, if $r = \gamma
F(x)$,  it can be shown that
\eqref{eq:proxproj} represents the standard gradient projection.
Recent work~\cite{proxlan12} proposes prox generalizations to the extragradient
scheme for monotone \vs{as well as pseudomonotone} deterministic variational \vs{inequality} problems \vs{and forms the inspiration for our analysis.} \vs{In monotone regimes, stochastic}
variants to these prox schemes have been suggested
in~\cite{juditsky2011} \vs{while inexact oracles regimes have also been recently examined in~\cite{dvur2018}}. However, those settings derive error
bounds under a monotone setting.  We consider a prox-based
generalization of the extragradient scheme for stochastic variational
problems (referred to as mirror-prox in~\cite{juditsky2011}). Our
contribution lies in showing that the sequence of iterates converges a.s.
to the solution set under pseudomonotone settings  as shown earlier
under appropriate choices of steplengths. Formally, the mirror prox
stochastic approximation (\ref{eq:extraproxstoch}) scheme is defined as follows:
\begin{align} \tag{MPSA}
\begin{aligned}
x_{k+1/2} & := \P_{X}(x_k,\gamma_k F(x_k;\omega_k)), \qquad \quad  \quad & k \geq 0 \\
x_{k+1}   & := \P_{X}(x_k,\gamma_k F(x_{k+1/2};\omega_{k+1/2})), \qquad & k
\geq 0.
\label{eq:extraproxstoch}
\end{aligned}
\end{align}
From the strong convexity of $s(x)$, it can be seen that there
exists a positive scalar $\theta \geq 1$ such that
\begin{align}
V(x,z) \geq \frac{\theta}{2}\|x-z \|^2.
\label{eq:Vstrong}
\end{align}
{Next, we recall the definition of the dual norm. 
\begin{align}
\|x\|_{*} = \textrm{sup} \left\{ z^{T}x \mid \|z \| \leq 1 \right\}.
\label{eq:dualnorm}
\end{align}
Based on the dual norm, we provide a modified statement for Assumptions
	(A2) and (A5), given by the following.
\begin{assumption}[Dual Norms]\em
\be
\item[]
\item[(A6)] For any $x,y \in X$, there exist $L_*$ and $B_*$ such that $\|F(x)-F(y)\|_{*}\leq L_{*}\|x-y\|$  and $\|F(x)\|_{*} \leq \frac{B_{*}}{2}.$
\item[(A7)] The conditional second moment of the error is bounded with
respect to the dual norm as specified by  $\bkE \left[ \| w_k\|_{*}^2
\mid
\mathcal{F}_k \right] \leq  \nu_*^2$ and 
$\bkE \left[ \| w_{k+1/2}\|_{*}^2 \mid \mathcal{F}_{k+1/2} \right] \leq  \nu_*^2.$
\ee
\end{assumption}
Under the assumption that $\nabla s(x)$ is Lipschitz continuous with a
positive constant $L_V$, the following holds ~\cite[Lemma 1.2.3]{nesterov04}:
\begin{align}
V(x,z) \leq  L_{V}^2 \|x-z \|^2.
\label{eq:Vlipsc}
\end{align}
We use the following result from~\cite{proxlan12}.
\begin{lemma}\em
\label{lemm:descproxfunc}
{Let $X \subseteq \mathbb{R}^n$ be a convex set and $p: X \rightarrow
	\mathbb{R}$ be a differentiable convex function. If $u^*$ is an optimal solution of $\min \{ p(u) + V (\tilde{x}; u) :
	u \in X \}$, the following holds}:
$$p(u^*)  + V (u^*; u) + V (\tilde{x}; {u^*}) \leq p(u) + V (\tilde{x}; u)
	\mbox{ for all } u \in X.$$
\end{lemma}
The next lemma relates prox functions over successive iterates and is
the generalization of Lemma~\ref{lemm:descexg} from standard extragradient
	schemes to the mirror-prox regime.
\begin{lemma}\em
Consider SVI$(X,F)$ and suppose Assumption (A6) holds.
Suppose $x^*$ is any solution of SVI$(X,F)$.
Consider the iterates generated by~(\ref{eq:extraproxstoch}). 
Then the following holds for any $k$:
\begin{align*}
 V(x_{k+1},x^*) &\leq \uvs{\left(1 + \frac{2\gamma_k^2}{\theta \beta}\right)} V(x_k,x^*) - \left(\frac{\theta}{2}- \frac{1}{c} \right) \|\uvs{x_{k+1}}-x_k\|^2 - 2\gamma_k w_{k+1/2}^{T}(x_{k}-x^*)  - u_k + \gamma_k^2 t_k,
\end{align*}
where $u_k = \gamma_k  F(x_{k})^{T}(x_{k}-x^{*})$, $c$ is a positive
scalar, and $t_k$ is defined accordingly.
\label{lemm:descproxexg}
\end{lemma}
\begin{proof}
{From the definition of the iterates, with $x_{k+1}$ being the solution
to the second prox-subproblem~\eqref{eq:extraproxstoch} and using Lemma~\ref{lemm:descproxfunc},
	{we obtain that for all $x \in X$,  
\begin{align}
V(x_k,x) \geq 
\gamma_k F(x_{k+1/2},\omega_{k+1/2})^{T}(x_{k+1}-x) + V(x_k,x_{k+1}) + V(x_{k+1},x) ,
\label{eq:ineq2}
\end{align}}
where $\gamma_k F(x_{k+1/2},\omega_{k+1/2})^{T}(x_{k+1}-x)$ takes the form of $p(u)$ defined earlier.
Adding and subtracting $x_{k}$ {from} the first term on the right hand side, we have
\begin{align}
  V(x_k,x)  & \geq \gamma_k F(x_{k+1/2};\omega_{k+1/2})^{T}(x_{k}-x) +
	 \gamma_k F(x_{k+1/2};\omega_{k+1/2})^{T}(x_{k+1}-x_{k}) \cr  &
+ V(x_{k},x_{k+1}) + V(x_{k+1},x^*).
\label{eq:ineqq}
\end{align}
Setting $x = x^*$, adding and subtracting $\gamma_k
	\vs{F(x_k)}^{T}(x_k-x^*)$ from the first term on the right, \eqref{eq:ineqq} can be rewritten as
follows.
\begin{align}
  V(x_k,x^*)  & \geq \gamma_k F(x_{k})^{T}(x_{k}-x^*) + \gamma_k w_{k+1/2}^{T}(x_{k}-x^*)
+ \gamma_k \left(F(x_{k+1/2})-F(x_{k})\right)^{T}(x_{k}-x^*)	\nonumber \\
& + \gamma_k F(x_{k+1/2};\omega_{k+1/2})^{T}(x_{k+1}-x_{k})
+ V(x_{k},x_{k+1}) + V(x_{k+1},\us{x^*}).
\label{eq:precauchy}
\end{align}
Rearranging and completing squares, for any positive $\beta$ and $c$,
			the following holds:
\begin{align}
&  V(x_k,x^*) + \frac{\gamma_k^2}{\beta}\|x_k-x^* \|^2 + \beta {\|
	F(x_{k+1/2})-F(x_{k})\|_*}^2	
 + c\gamma_k^2 {\| F(x_{k+1/2};\omega_{k+1/2}) \|_*}^2  + \frac{1}{c}\| x_{k+1}-x_{k} \|^2  \nonumber \\
&  \geq u_k  + \gamma_k w_{k+1/2}^{T}(x_{k}-x^*) + V(x_{k},x_{k+1}) + V(x_{k+1},{x^*}),
\label{eq:precauchy}
\end{align}
where $u_k = \gamma_k F(x_k)^T(x_{k}-x^*)$. By invoking the
	property $V(x_{k},x_{k+1}) \geq \frac{\theta}{2}\|x_k-x_{k+1} \|^2$, we have}
\begin{align}
 V(x_{k+1},x^*)
&\leq \left(1+\frac{2\gamma_k^2}{\theta \beta} \right) V(x_k,x^*) - \left(\frac{\theta}{2}- \frac{1}{c} \right) \|x_{k+1}-x_k \|^2 - \gamma_k w_{k+1/2}^{T}(x_{k}-x^*)  - u_k\nonumber \\
&    + \beta \uvs{L_*}^2 \|x_{k+1/2}-x_k \|^2 + c \gamma_k^2 \left(
		\|F(x_{k+1/2}) \|_*^2 + \| w_{k+1/2} \|_*^2 + 2w_{k+1/2}^T F(x_{k+1/2}) \right).
\label{eq:proxdecII}
\end{align}
From the definition of the iterates, with $x_{k+1/2}$ being the solution
to the first prox-subproblem~\eqref{eq:extraproxstoch}, from
	Lemma~\ref{lemm:descproxfunc} for any $\tilde x \in X$, we have that
$$ \gamma_k x_{k+1/2}^{T}\left( F(x_k)+w_{k} \right) +
V(x_{k},x_{k+1/2}) + V(\tilde x, x_{k+1/2})\leq \gamma_k
x_{k}^{T}F(x_{k}+ w_k) + {V(\tilde x,x_{k})}.$$
By choosing $\tilde x = x_k$, we obtain that 
$$ \gamma_k x_{k+1/2}^{T}\left( F(x_k)+w_{k} \right) +
V(x_{k},x_{k+1/2})  \leq \gamma_k
x_{k}^{T}F(x_{k}+ w_k) + {V(x_k,x_{k})} =  \gamma_k
x_{k}^{T}F(x_{k}+ w_k),$$
since $V(x_k,x_k) = 0$ and $V(x_k,x_{k+1/2}) \geq 0$. A consequence of
this inequality is that 
\begin{align}
\frac{\theta}{2} \| x_{k+1/2} - x_k\|^2 &\leq V(x_k,x_{k+1/2}) \leq \gamma_k (x_k-x_{k+1/2})^T \left( F(x_k)+w_{k} \right).
\label{eq:leftright} \\
\implies  \frac{\theta}{2} \| x_{k+1/2} - x_k\|^2  & \leq \frac{1}{2}\|
x_k-x_{k+1/2}\|^2 +  \frac{\gamma_k^2}{2}\| F(x_k)+w_{k} \|_*^2 \notag \\
\Longrightarrow \| x_{k+1/2} - x_k\|^2 & \leq \gamma_k^2 \frac{{\|
	F(x_k)+w_{k} \|_*^2}}{\theta-1}  { \ = \ }
	\frac{\gamma_k^2}{\theta-1} \left( {\|F(x_k) \|_*^2} + {\| w_k
			\|_*^2} +2
			F(x_k)^Tw_k \right) \notag \\
 & \leq \frac{\gamma_k^2}{\theta-1} \left( \frac{{B_*^2}}{4} + {\| w_k
		 \|_*^2}
		 +2 F(x_k)^Tw_k \right).\notag 
\end{align}
{Substituting} the above expression in~(\ref{eq:proxdecII}), our claim follows.
\begin{align}
 V(x_{k+1},x^*) & \leq \left(1+\frac{2\gamma_k^2}{\theta \beta} \right) V(x_k,x^*) - \left(\frac{\theta}{2}- \frac{1}{c} \right) \|x_{k+1}-x_k \|^2 - \gamma_k w_{k+1/2}^{T}(x_{k}-x^*)  - u_k \nonumber \\
    & +  \gamma_k^2 \left( c{\|F(x_{k+1/2}) \|_*^2 + c\| w_{k+1/2}
			\|_*^2} + 2cw_{k+1/2}^T F(x_{k+1/2}) \right) \nonumber \\
    & +   \frac{\beta {L_*}^2\gamma_k^2}{\theta-1} \left(
			\frac{{B_*}^2}{4} + \| w_k \|^2 +2 F(x_k)^Tw_k \right).
\label{eq:proxdeclast}
\end{align}
\end{proof}
{It can be observed that we may recover the statement of
	Lemma~\ref{lemm:descexg} by choosing $V(x,z)$ as the squared
		Euclidean norm.} {We now proceed to use \eqref{eq:proxdecII} to prove the almost sure
	convergence of the sequence produced by the (MPSA) scheme. }
\begin{proposition} [{\bf a.s. convergence of MPSA scheme for
	pseudomonotone-plus mappings}] \em
Consider the $\textrm{SVI}(X,F)$ and let $F$ be a pseudomonotone-plus map on $X$.  Let assumptions (A1),\uvs{(A2)}, (A3), (A4),
	(A6) and (A7) hold.
Additionally let $\beta>0$ and $c>0$ be chosen such that ${\theta \slash 2}-(1\slash c) \geq 0$.
Let $\{x_{k}\}$ denote a sequence of iterates generated by~\eqref{eq:extraproxstoch} and suppose $X^*$ {denotes}
the set of solutions to the SVI$(X,F)$. 
Then \uvs{$\{x_k\}$ is bounded a.s. and any limit point of $\{x_k\}$ is a solution of SVI$(X,F)$ \vvs{in an a.s. sense.}} 
\label{prop:proxasure}
\end{proposition}
\begin{proof}
{Considering~(\ref{eq:proxdeclast}) and by taking conditional expectations
with respect to $\mathcal{F}_{k}$ and using the tower law,
\begin{align*}
  \mathbb{E}[V(x_{k+1},x^*) \mid
	\mathcal{F}_{k}] & \leq \left(1+\frac{2\gamma_k^2}{\theta \beta} \right) V(x_k,x^*) - \gamma_k \overbrace{\bkE[\mathbb{E}[w_{k+1/2}^{T}(x_{k}-x^*)|\mathcal{F}_{k+1/2}]|\mathcal{F}_{k}]}^{\vs{\ =\ 0}}  - u_k \nonumber \\
    & + \gamma_k^2 \left( c \|F(x_{k+1/2}) \|^2 + c \bkE[\bkE[\|
			w_{k+1/2}\|^2 |\mathcal{F}_{k+1/2} ] \mid \mathcal{F}_{k}] \right) \nonumber \\
   & + 2\gamma_k^2c \left(\bkE[\overbrace{\bkE[w_{k+1/2}^T
		   F(x_{k+1/2})\mid \mathcal{F}_{k+1/2}]}^{\vs{ \ = \ 0}}\mid \mathcal{F}_{k}] \right) \nonumber \\
    & + \frac{\gamma_k^2 \beta {L_*}^2}{\theta-1} \left(
			\frac{{B_*}^2}{4} + \bkE[\| w_k \|_*^2 \mid
			\mathcal{F}_k] +2 \overbrace{\bkE[F(x_k)^Tw_k \mid \mathcal{F}_k]}^{\vs{\ =\ 0}} \right).
\end{align*}
Note that the term $-((\theta \slash 2) - (1 \slash c) )\|x_{k}-x_{k+1/2} \|^2 \leq 0$ by the definition of $\theta$ and $c$ and hence is dropped from the right hand side of the above expression. Invoking assumptions (A1-A2), we have
\begin{align*}
\mathbb{E}[V(x_{k+1},x^*) \mid
\mathcal{F}_{k}] & \leq \left(1+\frac{2\gamma_k^2}{\theta \beta} \right)
V(x_k,x^*) - u_k + \overbrace{ \gamma_k^2 \left(\frac{{B_*^2} +
		4{\nu_*^2}}{4}\right) \left(\frac{c(\theta-1)+\beta {L_*^2}}{\theta-1}\right)}^{s_k}.
\end{align*}}
{This inequality is analogous to \eqref{pseudo-ineq} in Prop.~\ref{prop:pseudo} and the remainder of
the proof follows by replacing {${1\over 2}\|x-y\|^2$} by $V(x,y)$.}
\end{proof}

\begin{corollary}[{\bf a.s. convergence  of MPSA scheme {under
	sub-classes of non-monotonicity}}] \em
Consider the $\textrm{SVI}(X,F)$. Let assumptions (A1), \uvs{(A2)}, (A3), (A4), (A6), and (A7) hold.
Consider a sequence of iterates $\{x_k\}$ generated by the MPSA
scheme where $\gamma_0$ {is chosen to be sufficiently small}. Suppose
one of the following statements hold:
\begin{enumerate}
\item[(i)]  $F$ satisfies the acute angle relation at $X^*$ given by~\eqref{eq:spseudo}.
\item[(ii)] $F$ is strictly pseudomonotone on $X$.
\item[(iii)]  $F$ is strongly pseudomonotone on $X$.
\item[(iv)]   $F$ is pseudomonotone on $X$  and  is given by the gradient of
$\mathbb{E}[f(x,\omega)]$.
\end{enumerate}
Then \uvs{$\{x_k\}$ is bounded a.s. and any limit point of $\{x_k\}$ is a solution of SVI$(X,F)$ \vvs{in an a.s. sense.}} 
\label{prop:pseudoadd_mpsa}
\end{corollary}
\begin{proof}
\label{eq:precauchy}
We begin by restating~(\ref{eq:proxdecII})
\begin{align*}
  V(x_k,x^*) &  \geq \gamma_k F(x_{k})^{T}(x_{k}-x^*) + \gamma_k w_{k+1/2}^{T}(x_{k}-x^*)
+ \gamma_k \left(F(x_{k+1/2})-F(x_{k})\right)^{T}(x_{k}-x^*)	\\
& + \gamma_k F(x_{k+1/2};\omega_{k+1/2})^{T}(x_{k+1}-x_{k})
+ V(x_{k},x_{k+1}) + V(x_{k+1},x).
\end{align*}
\vs{Since} $u_k = \gamma_k F(x_k)^T(x_{k}-x^*)$, \vs{we obtain the following bound by  using Young's} inequality {to derive an upper bound on $\gamma_k
	F(x_{k+1/2};\omega_{k+1/2})^{T}(x_{k+1}-x_{k})$} \vs{and completing
	squares.}
\begin{align}
&  V(x_k,x^*) + \frac{\gamma_k^2}{\beta}\|x_k-x^* \|^2 + \beta \| F(x_{k+1/2})-F(x_{k})\|_{*}^2	
 + c\gamma_k^2 \| F(x_{k+1/2};\omega_{k+1/2}) \|_{*}^2  + \frac{1}{c}\| x_{k+1}-x_{k} \|^2  \nonumber \\
&  \geq u_k  + \gamma_k w_{k+1/2}^{T}(x_{k}-x^*) + V(x_{k+1},\uvs{x^*}) + V(x_{k},x_{k+1}),
\label{eq:precauchy1}
\end{align}
\vs{Since}  $\|F(x_{k+1/2};{\omega}_{k+1/2})\|_{*}^2 \leq 2\|F(x_{k+1/2})\|_{*}^2 + 2\|w_{k+1/2} \|_{*}^2$ and \vs{proceedingin a similar fashion to} Lemma~\ref{lemm:descproxexg}, \vs{we have} that
\begin{align}
 V(x_{k+1},x^*) & \leq \left(1+\frac{2\gamma_k^2}{\theta \beta} \right) V(x_k,x^*) - \left(\frac{\theta}{2}- \frac{1}{c} \right) \|x_{k+1}-x_k \|^2 - \gamma_k w_{k+1/2}^{T}(x_{k}-x^*)  - u_k \nonumber \\
    & +  2\gamma_k^2 c\left( \|F(x_{k+1/2}) \|_{*}^2 + \| w_{k+1/2} \|_{*}^2 \right)
  + \gamma_k^2 \beta L_{*}^2 \left( \| x_{k+1/2}-x_{k} \|^2 \right).
\end{align}
\vs{By proceeding in a similar fashion to} Lemma~\ref{lemm:descproxexg}, \vs{we obtain the following.}
\begin{align*}
\frac{\theta}{2} \| x_{k+1/2} - x_k\|^2 & \leq V(x_k,x_{k+1/2}) \leq \gamma_k (x_k-x_{k+1/2})^T \left( F(x_k)+w_{k} \right) \\
& \leq \gamma_k \|x_{k}-x_{k+1/2}\| \|F(x_k)+w_k \|_{*}  
 \leq \frac{1}{2}\| x_k-x_{k+1/2}\|^2 +  \frac{\gamma_k^2}{2}\|
 F(x_k)+w_{k} \|_{*}^2 \\
\implies \| x_{k+1/2} - x_k\|^2  & \leq \frac{\|F(x_k)+w_k \|_{*}^{2}}{\theta-1}  \leq \frac{2\|F(x_k) \|_{*}^2 + 2 \|w_k \|_{*}^2}{\theta-1}.
\end{align*}
\uvs{Proceeding analogously to Lemma}~\ref{lemm:descproxexg}, \uvs{we have that}
\begin{align}
 V(x_{k+1},x^*) & \leq \left(1+\frac{2\gamma_k^2}{\theta \beta} \right) V(x_k,x^*) - \left(\frac{\theta}{2}- \frac{1}{c} \right) \|x_{k+1}-x_k \|^2 - \gamma_k w_{k+1/2}^{T}(x_{k}-x^*)  - u_k \nonumber \\
    & +    2c \gamma_k^2\left( \frac{B_{*}^2}{4} + \|w_{k+1/2} \|_{*}^2
			\right) + \frac{2\beta L_{*}^2\gamma_k^2}{\theta-1} \left( \frac{B_{*}^2}{4} + \| w_k \|_{*}^2  \right).
\label{eq:proxdecdualnorm}
\end{align}

For the case of pseudomonotone-plus regimes, the proof follows along the lines of Proposition~\ref{prop:proxasure}. The other non-monotone extensions mentioned above are straightforward and we do not provide the corresponding proofs.
\end{proof}
Without loss of generality, the assumptions (A6) and (A7) on the dual norm $\|.\|_{*}$ above can be replaced with those ((A2) and (A5)) \vs{on} the standard norm $\|.\|$ and the distance function $s(x)$ can be defined with respect to the dual norm. The above proofs for MPSA follow since $\|.\|_{**} = \|.\|$.
}

\section{Optimal rate statements}\label{sec:rate}
In the prior section, we proved the a.s. convergence of iterates
generated by the ESA and the MPSA schemes. In this section, we consider
the development of error bounds. Rate statements for the gap function
have been provided in the context of monotone stochastic variational
inequality problems by Tauvel et al.~\cite{juditsky2011}.  Here, we
generalize these findings in deriving rate statements under \ka{two requirements, strong
pseudomonotonicity and mere monotonicity with a weak sharpness requirement.} In the remainder of this section, we assume that the
steplength sequence is given by
\begin{align}
\gamma_k = \frac{\gamma_0}{k},
\label{eq:steplengthseq}
\end{align}
where $\gamma_0$ is a finite scalar. It is easy to observe that this
satisfies assumptions (A3)--(A4). 

\begin{proposition}[{\bf Rate statements under strong
	pseudomonotonicity}]\em
\label{prop:strongpseudo}
Suppose assumptions (A1)--(A5) hold. 
Let $F$ be \uvs{$\sigma-$}strongly pseudomonotone over $X$ and let the sequence of
iterates $\left\{x_k \right\}$ be generated by~(\ref{eq:ESA}).
Additionally let $X$ be compact such that for \vs{all $x \in X$, $\| x\|
\leq U$}, where $U$ is {a positive} constant. If $x^*$ denotes a solution
to the SVI$(X,F)$ and $M_{B}$ and $M_{\nu}$ are appropriately defined
constants, then the following hold:
\begin{enumerate}
\item[(a)] For any \vvs{$k > 0$}, we have that  
$$\mathbb{E}\left[\|x_{\vvs{k}} - x^*\|^2 \right] \leq\frac{M(\gamma_0)}{\vvs{k}},
	$$
	where $M(\gamma_0)$ is a suitably defined positive  scalar. 
\item[(b)] In addition, if $\gamma_0 = (2-\epsilon)/(2\sigma)$ where
$\epsilon \in (0,1/2)$, we have that
$$\mathbb{E}\left[\|x_{\vvs{k}} - x^*\|^2 \right] \leq
\frac{1}{\vvs{k}}\max \left\{
\frac{7(M_{B}+M_{\nu})}{4\sigma^2},\|x_0-x^*\|^2\right\}.$$
\end{enumerate}
\end{proposition}
\begin{proof}
\noindent {\bf (a).} We begin by considering (\ref{lip-bound-rec}) as follows:
\begin{align*}
 \|x_{k+1}-x^*\|^2 & \leq \left( 1+\frac{\gamma_k^2}{\beta}\right)
	\|x_{k} - x^*\|^2 - \overbrace{2\gamma_k(x_k-x^*)^TF(x_k)}^{\
		\triangleq \  u_k}
 + 2\gamma_k(x^*-x_{k})^Tw_{k+1/2} \\
& + \gamma^2_k \overbrace{\left( \frac{B^2(1+\beta L^2)}{4} + \| w_{k+1/2} \|^2  + \beta L^2 \| w_k \|^2
+ 2\beta L^2 F(x_k)^{T} w_k + 2w_{k+1/2}^T F(x_{k+1/2}) \right)}^{\
	\triangleq \  t_k}.
\end{align*}
{Since $x^*$ is a solution of VI$(X,F)$, we have that for any
	feasible $x_k$,  $F(x^*)^T(x_{k}-x^*)
\geq 0$. {By recalling} the definition of strong pseudomonotonicity,
\begin{align}
-u_k & = -2 \gamma_k F(x_{k})^T(x_{k}-x^*)
 \leq -2 \gamma_k \sigma \|x_{k}-x^*\|^2.
\label{eq:ukdescrate}
\end{align}}
{Employing the bound~(\ref{eq:ukdescrate}) in~(\ref{lip-bound-rec}),
\begin{align}
 \quad \|x_{k+1}-x^*\|^2 & \leq \left( 1-2\sigma \gamma_k \right) \|x_{k} - x^*\|^2
 + 2\gamma_k(x^*-x_{k})^Tw_{k+1/2} + \frac{\gamma_k^2}{\beta} \| x_{k} - x^*\|^2 \nonumber \\
& + \gamma^2_k \left( \frac{B^2(1+\beta L^2)}{4} + \| w_{k+1/2} \|^2
		\right. \notag \\
	& \left. + \beta L^2 \| w_k \|^2
+ 2\beta L^2 F(x_k)^{T} w_k + 2w_{k+1/2}^T F(x_{k+1/2}) \right).
\label{eq:ratedescI}
\end{align}}
{Taking expectations on both sides of~(\ref{eq:ratedescI}),
\begin{align*}
\bkE[\|x_{k+1}-x^{*}\|^2]  & \leq
\underbrace{\bkE[(1-2\sigma\gamma_k)\|x_{k}-x^{*}\|^2]}_{\text{{
	\tiny Term A}}} + \underbrace{\gamma_k^2\bkE\left[
		\frac{B^2(1+\beta L^2)}{4}+\frac{\|x_k-x^* \|^2}{\beta}+2\gamma_k(x^{*}-x_{k})^{T}w_{k+1/2}\right]}_{\text{\tiny Term B}} \\
&  + \uvs{\gamma_k^2}\underbrace{\bkE\left[\| w_{k+1/2} \|^2  + \beta L^2 \| w_k \|^2
+ 2\beta L^2 F(x_k)^{T} w_k + 2w_{k+1/2}^T F(x_{k+1/2})\right]}_{\text{{\tiny Term C}}}.
\end{align*}}
We begin by deriving a bound on Term C by noting that the
	conditional first moments $\mathbb{E}[w_k \mid \Fscr_k] = 0$ and
		$\mathbb{E}[w_{k+1/2} \mid \Fscr_{k+1/2}] = 0$.
\begin{align*}
\mbox{Term C} & = \bkE\left[ \bkE \left[\| w_{k+1/2} \|^2 | \mathcal{F}_{k+1/2}\right]\right] + \beta L^2 \bkE\left[ \bkE \left[\| w_k \|^2 | \mathcal{F}_{k}\right]  \right]
+ {2\beta L^2 \bkE\left[ \bkE \left[ F(x_k)^{T} w_k | \mathcal{F}_{k}\right] \right]} \\
& + {\bkE\left[ \bkE \left[ 2w_{k+1/2}^T F(x_{k+1/2}) | \mathcal{F}_{k+1/2}\right] \right]}  \leq \gamma_k^2 (1+\beta L^2) \nu^2 \vvs{ \ \triangleq \ c(\beta) }.
\end{align*}
{Next, Term B can be bounded as follows:}
\begin{align*}
\mbox{Term B} & = \gamma_k^2 \frac{B^2(1+\beta L^2)}{4}+
\gamma_k^2 \bkE\left[ \frac{\|x_k-x^* \|^2}{\beta} \right] +2 \gamma_k \bkE\left[ {\bkE \left[(x^{*}-x_{k})^{T}w_{k+1/2}| \mathcal{F}_{k+1/2} \right]}\right] \\
& \leq \gamma_k^2 \left( \frac{B^2(1+\beta L^2)}{4} + \frac{4U^2}{\beta} \right) \vvs{ \ \triangleq \ b(\beta) }.
\end{align*}
{To minimize \vvs{$c(\beta) + b(\beta)$}, it suffices to minimize the following expression.
$$ s_0(\beta) = \beta \left( \frac{B^2L^2}{4}+L^2\nu^2 \right) + \frac{4U^2}{\beta}.$$
Setting  $\beta^* =  \frac{4U}{L\sqrt{B^2+4\nu^2}}$ by minimizing the above expression, we obtain that  
$$\bkE[\|x_{k+1}-x^{*}\|^2] \leq (1-q_k) \bkE[\|x_{k}-x^{*}\|^2] + t_k,$$
\begin{align*}
\textrm{where} \hspace{2mm} q_k = \frac{2\sigma\gamma_0}{k}, \hspace{2mm} t_k = \frac{\gamma_0^2
(M_{\nu}+M_B)}{k^2}, \hspace{2mm} M_B = \frac{B^2}{4},
\hspace{2mm} \textrm{and} \hspace{1mm} M_{\nu} = \nu^2+2UL\sqrt{B^2+4\nu^2}.
\end{align*}}
By assuming that $2\sigma\gamma_0 > 1$ and by invoking
Lemma~\ref{ind_lemma} presented in the appendix, we obtain that
$$ \bkE \left[\|x_{k}-x^{*}\|^2 \right] \leq \frac{{M(\gamma_0)}}{k}, \hspace{1mm}
\mbox{ where } {M(\gamma_0)} \triangleq \max
\left\{\frac{
	\gamma_0^2(M_{\nu}+M_B)}{2\sigma\gamma_0-\lfloor 2 \sigma
		\gamma_0\rfloor },\|x_{0}-x^{*}\|^2\right\}. $$ 

\noindent {\bf (b). } Recall  that $M(\gamma_0)$ may be  defined as follows:
$$M(\gamma_0) = \max\left\{t_{0}(\gamma_0) (M_{\nu}+M_{B}), \|x_0-x^*
\|^2\right\}, \hspace{0.5mm} \text{where} \hspace{0.5mm} t_0(\gamma_0) = \left( \frac{\gamma_0^2}{2\sigma \gamma_0 - \lfloor 2 \sigma \gamma_0 \rfloor} \right).$$
Based on Lemma~\ref{lemm:t0}, an $\epsilon$-infimum of $t_0(\gamma_0)$ is
achieved by choosing $\gamma_0 = \vs{\gamma}$, \vs{where $\gamma = (2-\epsilon)/(2\sigma)$ and } $\epsilon \in (0,1/2)$.
\vs{Since $\lfloor 2\sigma \gamma\rfloor = 1$, it} follows that 
$t_0(\gamma) = \frac{(2-\epsilon)^2}{4\sigma^2
	\left(1-\epsilon\right)}$, which may be bounded as follows when
	$\epsilon \in (0,1/2)$:
	\begin{align*}
t_0\left(\frac{(2-\epsilon)}{(2\sigma)}\right) = \frac{(2-\epsilon)^2}{4\sigma^2
	\left(1-\epsilon\right)} & = \frac{(1-\epsilon)^2}{4\sigma^2
	\left(1-\epsilon\right)} + \frac{3-2\epsilon}{4\sigma^2
	\left(1-\epsilon\right)} 
			 \leq \frac{(1-\epsilon)}{4\sigma^2
	} + \frac{3}{2\sigma^2} 
		 \leq \frac{7}{4\sigma^2}. 
\end{align*}
		 It follows that 
\begin{align*}
\bkE[\|x_k-x^* \|^2] & \leq \frac{1}{k}\max
	\left\{\frac{7(M_{B}+M_{\nu})}{4\sigma^2},  \|x_0-x^* \|^2 \right\}.
\end{align*}
\end{proof}
\noindent {\bf Remark:} This result is notable from several standpoints. First, in contrast
	to rate statements for settings that lack strong monotonicity, we
		provide a rate statement in terms of solution iterates, rather
		than in terms of the gap function. Second, our rate statement is
		optimal from a rate standpoint with a slightly poorer constant, in part due to the use of $B^2$ instead of $B^2/4$.
Notably, in strongly convex optimization problems (cf.~\cite{SPbook}), it can be
	seen that based on the optimal initial steplength, we have that
	$\mathbb{E}[\|x_K-x^*\|]^2 \leq {1\over K}\max\left(\frac{2(B^2+\nu^2)}{\sigma^2},
			\|x_0-x^*\|^2\right).$
Next, we generalize the above rate result to prox functions with general
distance functions.




\begin{corollary} [{\bf Rate statements for MPSA}] \em
\label{cor:strongpseudo}
Suppose assumptions (A1),\uvs{(A2)}, (A3, (A4), (A6), and -(A7) hold. 
Let $F$ be \uvs{$\sigma-$}strongly pseudomonotone over $X$ and let the sequence of
iterates $\left\{x_k \right\}$ be generated by~(\ref{eq:extraproxstoch}).
Additionally let $X$ be compact such that for \vs{all $x \in X$, $\| x\|
\leq U$}, where $U$ is {a positive} constant. If $x^*$ denotes a solution
to the SVI$(X,F)$ and $M^*_{B}$ and $M^*_{\nu}$ are appropriately defined
constants, then the following hold:
\begin{enumerate}
\item[(a)] For any \vvs{$k > 0$}, we have that  
$$\mathbb{E}\left[\|x_{\vvs{k}} - x^*\|^2 \right] \leq\frac{M(\gamma_0)}{\vvs{k}},
	$$
	where $M(\gamma_0)$ is a suitably defined positive  scalar. 
\item[(b)] In addition, if $\tilde \sigma = \tfrac{\sigma}{L_V^2}$ and $\gamma_0 = (2-\epsilon)/(2{\tilde \sigma})$,  where
$\epsilon \in (0,1/2)$, we have that
$$\mathbb{E}\left[\|\vvs{x_{k}} - x^*\|^2 \right] \leq
\vvs{\frac{1}{k}}\max \left\{
\frac{7(M^*_{B}+M^*_{\nu})}{4{\tilde \sigma}^2},\|x_0-x^*\|^2\right\}.$$
\end{enumerate}
\end{corollary}
\begin{proof}
{We begin by recalling the inequality given by ~(\ref{eq:proxdeclast}):
\begin{align*}
 V(x_{k+1},x^*) & \leq \left(1+\frac{2\gamma_k^2}{\theta \beta} \right) V(x_k,x^*) - \left(\frac{\theta}{2}- \frac{1}{c} \right) \|x_{k+1}-x_k \|^2 - \gamma_k w_{k+1/2}^{T}(x_{k}-x^*)  - u_k \nonumber \\
    & +  \gamma_k^2 \left( c\|F(x_{k+1/2}) \|_*^2 + c\| w_{k+1/2} \|_*^2 + 2cw_{k+1/2}^T F(x_{k+1/2}) \right) \nonumber \\
    & + \gamma_k^2 \left( \frac{\beta L_*^2}{\theta-1} \left(
				\frac{B_*^2}{4} + \| w_k \|_*^2 +2 F(x_k)^Tw_k \right) \right).
\end{align*}}
Using $u_k \geq -2\gamma_k \sigma \|x_k-x^* \|^2$
	from~(\ref{eq:ukdescrate}) \vs{and by noting that
$V(x_k,x^*) \leq L_V^2\|x_k-x^* \|^2 \leq 4L_V^2 U^2$ based on compactness and Lipschitzian properties}, we have the following \vs{by taking expectations on both sides.}
\begin{align}
 & \quad \bkE[V(x_{k+1},x^*)]  \leq \left(1-\frac{2\sigma\gamma_k}{L_V^2} \right) \bkE[V(x_k,x^*)] - \underbrace{\left(\frac{\theta}{2}- \frac{1}{c} \right) \bkE[\|x_{k+1}-x_k \|^2]}_{\tiny \text{Term H}}  \nonumber \\
    & + \underbrace{\bkE \left[- \gamma_k w_{k+1/2}^{T}(x_{k}-x^*) +
		\gamma_k^2 \left( \frac{\beta L_*^2B_*^2}{4(\theta-1)} +
				\frac{8L_V^2U^2}{\theta \beta}\right)\right]}_{\tiny \text{Term B}} \nonumber \\
    & +  \underbrace{\gamma_k^2\bkE \left[\left( c\|F(x_{k+1/2}) \|_*^2 +
			c\| w_{k+1/2} \|_*^2 + 2cw_{k+1/2}^T F(x_{k+1/2})  +
			\frac{\beta L_*^2}{\theta-1}\left(\| w_k \|_*^2 +2 F(x_k)^Tw_k
				\right) \right)\right]}_{\tiny \text{Term C}}. \notag
\end{align}
{Choosing $c = 2 \slash \theta$  \vvs{allows from dropping Term H while} Terms B and C \vvs{can be bounded in a fashion similar to Proposition~\ref{prop:strongpseudo}}, \vvs{where}
we have $$ \mbox{\ Term B }\leq \gamma_k^2 \left( \frac{\beta L_{*}^2
		B_{*}^2}{4(\theta-1)} + \frac{8  L_V^2U^2}{\beta \theta} \right) \text{and}
\mbox{\ Term C }\leq \gamma_k^2 \left(\frac{B_*^2+4\nu^2}{2\theta}+ \frac{\beta L_{*}^2\nu^2}{\theta-1}\right).$$}
\uvs{Akin to} Proposition~\ref{prop:strongpseudo}, we minimize \vvs{the bounds on} Term B($\beta$) + Term C($\beta$) and it suffices to minimize $s_0(\beta)$, defined as 
$$s_0(\beta) = \beta \left( \frac{L_{*}^2 B_{*}^2}{4(\theta-1)} + \frac{L_{*}^2 \nu^2}{\theta-1} \right) + \frac{8L_V^2U^2}{\beta \theta}. $$
Noting that $\beta^* = \frac{4L_VU}{L_{*}} \sqrt{\frac{2(\theta-1)}{\theta (B^2+4\nu^2)}}$ minimizes \vvs{$s_0(\beta)$}, we have \vvs{that}
\ask{\begin{align*}
\bkE[V(x_{k+1},x^*)] & \leq \left(1-q_k\right) \bkE[ V(x_k,x^*)]
+ t_k, \quad \mbox{ where }	
\end{align*}
$$q_k \triangleq \frac{2\sigma\gamma_0}{L_V^2k} = \frac{2 {\tilde
\sigma} \gamma_0}{k} , t_k \triangleq
	\gamma_k^2 \left( M^*_B+M^*_{\nu} \right), M^*_{B} \triangleq  
		\frac{B_*^2}{2\theta}, \hspace{1mm}\textrm{and} \hspace{1mm}
	M^*_{\nu} \triangleq \left\{ \frac{2\nu^2}{\theta} + 2L_{*}UL_V \sqrt{\frac{2(B^2+4\nu^2)}{\theta(\theta-1)}} \right\} .$$}
By assuming that $2\sigma\gamma_0 > 1$ and by invoking
Lemma~\ref{ind_lemma} presented in the appendix, we obtain that
$$ \bkE \left[\|x_{k}-x^{*}\|^2 \right] \leq \frac{{M(\gamma_0)}}{k}, \hspace{1mm}
\mbox{ where } {M(\gamma_0)} \triangleq \max
\left\{\frac{
	\gamma_0^2(M^*_{\nu}+M^*_B)}{2{\tilde \sigma}\gamma_0-\lfloor 2 {\tilde
		\sigma}
		\gamma_0\rfloor },\|x_{0}-x^{*}\|^2\right\}. $$

\noindent {\bf (b).} From Proposition~\ref{prop:strongpseudo}(b), by setting $\gamma_0 =
(2-\epsilon)/(2{\tilde \sigma})$,  it follows that 
\begin{align*}
\bkE[\|x_k-x^* \|^2] & \leq \frac{1}{k}\max
	\left\{\frac{7(M^*_{B}+M^*_{\nu})}{4{\tilde \sigma}^2},  \|x_0-x^* \|^2 \right\}.
\end{align*}
\end{proof}
\ask{\noindent {\bf Remark:} It can be observed that when the Euclidean norm is used as the distance function
($\theta$ = 2 and $L_V = 1$), the
MPSA scheme reduces to the standard extragradient scheme and we obtain the same upper bound as with the case of \vs{ESA}.}

{We conclude this section with a rate analysis on the solution
	iterates under mere monotonicity of the map but under an additional
		requirement of weak-sharpness. We observe that the specification
		of the initial steplength requires globally minimizing a product
		of positive  functions over a Cartesian product of convex sets.
While there are settings where this product is indeed convex, \ka{it may also turn out to be nonconvex.}
Yet, we observe that  the global minimizer can be tractably obtained by solving two optimization
problems. Lemma~\ref{lemm:productfns} provides the necessary support for
this result. Note that in our setting, one \uvs{of these functions is a discontinuous
nonconvex function} and its infima are analyzed in  Lemma~\ref{lemm:t0}.}

{Next, under a monotonicity and weak-sharpness requirement,
   the {ESA} scheme is shown to display the optimal rate of convergence in
   solution iterates. Additionally, we prescribe the optimal initial
   steplength that minimizes the mean-squared error by deriving the
   global minimizer of a nonconvex function in closed-form.}

\begin{proposition}[{\bf Rate statement under monotonicity and weak sharpness}] \em
{Consider the SVI$(X,F)$. {Suppose assumptions (A1)--(A5) hold} and let
$\gamma_k$ be defined as per~(\ref{eq:steplengthseq}). 
Let $F(x)$ be a monotone map over the set $X$.} 
Let the mapping $F(x)$ and  solution set $X^*$ possess the
weak-sharpness property  with constant $\alpha$ and let $X$ be compact
such that $\| x\| \leq U$ for all $x \in X$. Suppose $x_k$ is generated
by ~(\ref{eq:ESA}). Then the following hold.
\begin{enumerate}
\item[(a)] For any \vvs{$k > 0$}, we have that
\begin{align*}
 \mathbb{E}[\textrm{dist}^2(x_{k},X^*)] \leq \frac{M(\gamma_0)}{k},
\end{align*}
where $M(\gamma_0)$ is a suitably defined scalar.
\item[(b)]  In addition, if $\gamma_0 = (2-\epsilon)/(2{\bar \sigma})$ and \uvs{$\bar \sigma \triangleq \tfrac{\alpha}{2U}$} where
$\epsilon \in (0,1/2)$ and $s_0^*$ is a suitably defined positive scalar, we have that
{\begin{align*}
 \mathbb{E}[\textrm{dist}^2(x_{k},X^*)] \leq \frac{1}{k}\max
 \left\{\frac{7 s^*_0}{4 \bar \sigma^2},\textrm{dist}^2(x_0,X^*)\right\}.
\end{align*}}
\end{enumerate}
\label{prop:rateexg}
\end{proposition}
\begin{proof}
\noindent {\bf (a).} By taking expectations on both sides of (\ref{lip-bound-rec}) and by leveraging the property of weak-sharpness, $u_k = -2\gamma_k F(x_{k})^T(x_{k}-x^*) \leq -2\gamma_k \alpha \textrm{dist}(x_k,X^*)$, we have
{\begin{align}\label{ineq_rate_wsprev}
\mathbb{E}[\| x_{k+1} -x^*\|^2]
& \leq  \bkE[\|x_k-x^* \|^2] -2\gamma_k \alpha \textrm{dist}(x_k,X^*) + \bkE[t_k] + \frac{\gamma_k^2}{\beta}\bkE[\|x_k-x^* \|^2].
\end{align}}
Since $\textrm{dist}^2(x_{k+1},X^*) \leq\| x_{k+1} -x^*\|^2$ and
$\|x_k-x^* \|^2 \leq 4U^2$, by minimizing the expression on the right of \eqref{ineq_rate_wsprev}
in $x^*$ over $X^*$, we have
{\begin{align}\label{ineq_rate_ws}
\mathbb{E}[\textrm{dist}^2(x_{k+1},X^*)]
& \leq \bkE[\textrm{dist}^2(x_k,X^*)] 
 - 2\gamma_k \alpha \mbox{dist}(x_k,X^*) + \bkE[t_k]+\frac{4\gamma_k^2U^2}{\beta}.
\end{align}}
{Since $\textrm{dist}(x_k,X^*) \leq 2U$, it follows that $- 2\gamma_k
\alpha \mbox{dist}(x_k,X^*) \leq -\frac{ \gamma_k
\alpha}{U} \mbox{dist}^2(x_k,X^*)$. Furthermore,
bounding $t_k$ along the lines of Proposition~\ref{prop:strongpseudo}, we have}
\begin{align*}
\mathbb{E}[\textrm{dist}^2(x_{k+1},X^*)]
& \leq \left(1-\frac{\alpha \gamma_k}{U}\right)
\bkE[\textrm{dist}^2(x_{k},X^*)] + \gamma_k^2 \left(M_{\nu} + M_{B}\right) 
 = (1-q_k)\bkE[\textrm{dist}^2(x_{k},X^*)] + s_k,
\end{align*}
where
\begin{align}
 \hfill M_{\nu}(\beta) & \triangleq (1+\beta L^2)\nu^2, M_{B}(\beta) \triangleq
 (1+\beta L^2)\frac{B^2}{4} + \frac{4U^2}{\beta},  \nonumber \\
 q_k & = \frac{2\bar \sigma\gamma_0}{k}, s_k = \frac{s_0(\beta)
	 \gamma_0^2}{k^2}, \bar \sigma 
 = \frac{\alpha}{2U},  \mbox{ and } s_0(\beta) = M_{C}(\beta)+M_{\nu}(\beta).
\label{eq:boundexgprox}
\end{align}
{Through the application of Lemma~\ref{ind_lemma}, we obtain the following
bound on mean-squared error for every positive integer $K$:}
{\begin{align*}
\mathbb{E}[\textrm{dist}^2(x_{K},X^*)] & \leq \frac{1}{K}\max
\left\{h(\gamma_0) s_0(\beta),\textrm{dist}^2(x_0,X^*)\right\},\\
\mbox{where} \hspace{2mm} h(\gamma_0) & \triangleq
\frac{\gamma_0^2}{2{\bar \sigma} \gamma_0 - \lfloor 2{\bar \sigma}
\gamma_0 \rfloor } \mbox{ and } s_0(\beta) \triangleq (M_{\nu}(\beta) +
				M_{B}(\beta)).
\end{align*}}

\noindent {\bf (b.)}
Suppose $\Gamma_0$ and $\cal Z$ are sets defined as
$$ \Gamma_0 \triangleq \left\{\gamma_0: 2 > 2{\bar \sigma} \gamma_0 > 1 \right\} \mbox{ and } {\cal Z} \triangleq \left\{\beta: \beta \geq 0 \right\}. $$
Moreover, $h(\gamma_0) s_0(\beta)$ is a product of two positive
functions and a global minimizer of this product
can be obtained by getting a global minimizer of each by invoking Lemma~\ref{lemm:productfns}.
Of these, an $\epsilon$-infimum of $h(\gamma_0)$ can be obtained as
$\gamma_0^* = (2-\epsilon)/(2\bar\sigma)$ where $\epsilon \in (0,1/2)$.
A minimizer  $\beta^*$ of the convex function $s_0(\beta)$ is given by the following.
{$$ \min_{\beta} \hspace{3mm}  \left( \left(1+\beta L^2\right)\left(\nu^2+\frac{B^2}{4}\right) + \frac{4U^2}{\beta}\right)
\implies \beta^* = \frac{4U}{L\sqrt{B^2+4\nu^2}}, $$
implying that 		
\begin{align*}
s_0^* = \frac{B^2}{4}+\nu^2 +2UL\sqrt{B^2+4\nu^2}. 
\end{align*}
We may then conclude that 
{\begin{align*}
 \mathbb{E}[\textrm{dist}^2(x_{K},X^*)] \leq \frac{1}{K}\max
 \left\{\frac{7 s^*_0}{4 \bar \sigma^2},\textrm{dist}^2(x_0,X^*)\right\}.
\end{align*}}}
\end{proof}

\noindent {\bf Remark:} In past research, the optimality of the rate of
convergence has been proved for monotone SVIs but in terms of the gap
function. Our result shows that under a suitable weak-sharpness
property,  rate optimality also holds in terms of the
solution iterates in a non-ergodic sense. Notably, we further refine the statement by
selecting the initial steplength by (globally) minimizing a nonconvex
function.\footnote{We prefer not to qualify the initial steplength as ``optimal'' since the error bound in general is a function of $\gamma_0$ and $\beta$.}

\section{Numerical Results}\label{sec:numbers}
{In this section, we examine the performance of the presented schemes on
a suite of four test problems described in Section~\ref{sec:test} while the
algorithm parameters are defined in Section~\ref{sec:algo}. In
Section~\ref{sec:perf}, we compare the performance of the ESA scheme
with the MPSA schemes over the suite of test problems. Finally in Section~\ref{sec:error}, we
compare the empirical rates with the theoretically predicted rates and
quantify the benefits of optimal initial steplength.}
\subsection{Test Suite}\label{sec:test}
The first two test problems are stochastic fractional convex quadratic
and  nonlinear, both of which lead to pseudomonotone stochastic
variational inequality problems.  The third set of test problems are
stochastic variational inequality problems that represent the
(sufficient) equilibrium conditions of a stochastic Nash-Cournot game.
\ka{While the players maximize pseudoconcave expectation-valued
functions, the resulting stochastic variational inequality problem
is \textit{not necessarily} pseudomonotone}. However, some choices of
parameters lead to pseudomonotone SVIs. Our fourth test problem is
Watson's complementarity problem~\cite{watsoncp79}, which is not
necessarily monotone.\\ \vspace{-0.1in}

\noindent {\em (i) Fractional Convex Quadratic Problems:} Maximizing or minimizing ratios in
engineering \uvs{settings often leads} to stochastic fractional convex problems of the form:
$\min_{x \in X} \bkE \left[f(x;\omega) \slash g(x)\right]$
where $\bkE[f(x;\omega)]$ and $g(x)$ are strictly positive convex quadratic and linear
functions, respectively, defined as
 $f(x;\omega) \triangleq 0.5x^T (\theta UU^T+ \lambda
		 V(\omega))x+ 0.5((c+\bar{c}(\omega))^Tx+4n)^2$ and
$g(x)   \triangleq r^Tx+t+4n.$
We note that $V({\omega})$ and $\bar{c}({\omega})$ are randomly
generated from standard normal and uniform distributions, $U$ and $c$
are deterministic constants
generated once from the standard normal distribution, while $r$ and
$t$ are generated once from uniform
distributions. We note that $\theta = 0.025$ and
$\lambda = \epsilon \|\theta UU^T\|_{F} \slash \|V({\omega})\|_{F}$,
where $\| .\|_{F}$ denotes the Frobenius norm and $\epsilon = 0.025$.
The set $X$ is defined as $X \triangleq \left\{ x \mid  Ax \leq v,
    0 \leq x \leq 4 \right\},$ where $A \in \mathbb{R}^{m
	   \times n}$ and $v \in \mathbb{R}^{m \times 1}$ are generated once from standard normal and uniform distributions
	   respectively. Note that $m = \lceil n \slash 10 \rceil$
	   is a variable dependent integer. It is easily seen that the
	   resulting SVI is pseudomonotone.\\ \vspace{-0.1in}

\noindent {\em (ii) Fractional Convex Nonlinear Problems: } We consider
a nonlinear variant of (i) with the same parameters and numerator but an
exponential denominator $g(x) = 10^4(\lambda-e^{(r^{T}x+t+4n) \slash
		2000})$, where $\lambda = e^{(8n+2) \slash 2000}$.\\ \vspace{-0.1in}

\noindent {\em (iii) Nash-Cournot Games:} {Next we consider a
	Nash-Cournot game with $n$ selfish players, all of which sell the same
commodity~\cite{kannanuday12siopt,econ11pseudo} at a price given by the
function of the aggregate sales as per the Cournot
specification~\cite{allaz93cournot,hobbs86mill}. Specifically, the $i$th agent solves the following problem: $\max_{x_i \in X_i}  \ f_i(x) = \bkE
[p(\bar{x};\omega)x_{i}],$ where
$p(\bar x;\omega) = (a - b^{\omega}\bar{x})^{\kappa}$,
 $\bar x= \sum_{i=1}^{n}x_{i}$,  $\kappa \in (0,1)$ and $X_i = \left\{
	 x_i \mid Ax \leq v, 0 \leq x_i \leq
3n \right\}.$ We note that $a = 100 \lceil n
\slash 3 \rceil$ while $b^{\omega}$ is generated from a
uniform distribution with mean $1$ and standard deviation $\epsilon$,
where $\epsilon = 0.025$. We note that $A$ and $v$ are also generated
randomly as stated earlier. The equilibrium of this shared  constraint
Nash game~\cite{facchinei07generalized} is given by a variational inequality
problem.  Note that agent payoffs are
pseudoconcave~\cite[Theorem 3.4]{econ11pseudo} and the (sufficient)
	equilibrium conditions are given by a variational inequality
	VI$(X,F)$, which is not necessarily pseudomonotone and
where $F(x) =\pmat{\nabla_{x_i} f_i(x)}_{i=1}^{n}$.}\\ \vspace{-0.1in}

\noindent {\em (iv) Watson's Problem:} Finally, we consider a stochastic
variant of the ten
variable non-monotone linear complementarity problem, first proposed by
Watson~\cite{watsoncp79}:
$0 \leq x \perp \bkE \left[\uvs{(M + \epsilon M^{\omega})x} + q + \epsilon
q^{\omega}\right]$  $\geq 0,$
where $M^{\omega}$ and $q^{\omega}$ are randomly generated matrices and
vectors (from the standard normal distribution) respectively and
$\epsilon = 0.025$ refers to the level of noise. We omit the definition
of the ten-dimensional matrix M, which can be found in ~\cite[Example
3]{watsoncp79}. Note that $q = e_i$ and we consider ten different
instances, each corresponding to a coordinate direction $e_i$.
\subsection{Algorithm parameters and termination criteria}\label{sec:algo}
We conduct two sets of tests, the first of these pertains to
the a.s. convergence behavior while
the second set compares the empirical rate estimates with the theoretically
prescribed levels. All the numerics were generated with
Matlab R2012a on a Linux OS with a  2.39 GHZ processor and
16 GB of memory. For the first two test problem sets, $x_0 = 2e$ and $\gamma_0$ is 1
and 2.5 respectively while for the second two test problem sets, $x_0 =
0$ and $\gamma_0$ is 2.5 and 0.6, respectively.\\ \vspace{-0.1in}

\noindent {(i) \em  a.s. convergence:} Here, $n$ was varied
from 10 to 30 in steps of 2 for the first three test problems while ten
different instances of $q$ were generated as stated earlier for the
Watson's problem, leading to a total of 40 test instances.  Recalling
that $x$ is a solution of VI$(X,F)$ if and only if $F^{\rm nat}_X(x) =
x-\Pi_X(x- F(x)) = 0$, a.s. convergence can be \ka{empirically verified
based on the value of
$\psi(x_k) = \| F^{\rm nat}_X(x_k)\| $}. Note that our problem choices allow for evaluating
the expectation, which is generally not possible in \uvs{stochastic}
regimes.\\ \vspace{-0.1in}

\noindent {(ii) \em Rate statements:} When evaluating the rate estimates, we
consider a modified Nash-Cournot game. The price was made affine, $\kappa
= 1$ and the linear constraints were dropped. We generated ten different problem instances for $n$
ranging from 10 to 19 and set $a = 0.1\lceil n
\slash 10 \rceil$ and $b = a\slash n$. Note that $b^{\omega}$ was
generated from a normal distribution with mean $b$ and standard
deviation $\epsilon$, where $\epsilon = 0.025b$. The associated set and
mapping are defined to be $F(x) = b(I + ee^T)x - ae, \hspace{2mm} X =
\left\{ x \mid 0 \leq x_i \leq 1,  i = 1, \hdots, n\right\},$ where $e$
and $I$ denote the
vector of ones and the identity
matrix. We note that $\nabla F(x) = b(I+ee^{T})$,
is strongly monotone (implies
strongly pseudomonotone) with constant
$\sigma = b$. The stochastic error can be bounded as follows: 
$$\mathbb{E}[\|F(x;\omega)-F(x) \|^2] =
	\mathbb{E}[|b-b^{\omega}|^2] \|(I+ee^T)x \|^2 \leq
n(n+1)^2\epsilon^2 = \nu^2.$$
Further, we have that
$\| F(x) \|  = \|b(I+ee^T)x-ae \| = a \left\|n(I+ee^T)s-e \right\|
\leq a\sqrt{n},$
where the last inequality follows from $b = a \slash n$ and $0 \leq s
\leq e$. This implies that  $B = 2a\sqrt{n}$.  Since $0 \leq x_i \leq 1$, 
it follows that $\| x\| \leq \sqrt{n} = U$. 
It is easy to observe that the Lipschitz constant $L = b\|I+ee^T \|_{\textrm{F}} = b\sqrt{n^2-n+4n} =a\sqrt{(n+3) \slash (n)}.$
If $x^*$ denotes the
unique solution of VI$(X,F)$, then the empirical error $\psi_e(x_K)$ and
theoretical error $\psi_b(x_K)$ are defined as follows (see
		Proposition~\ref{prop:strongpseudo}):
\begin{align}
\psi_{e}(x_K) = \frac{1}{N}\sum_{j =1}^N \|x_K^j-x^* \|^2, \hspace{2mm}
\psi_{b}(x_K) = \frac{{M(\gamma_0)}}{K} \geq
	\bkE[\|x_{K}-x^{*}\|^2],  \hspace{2mm} M(\gamma_0) = \frac{\gamma_0^2(M_{\nu}  + M_{B})}{2\sigma \gamma_0-1},
\label{eq:psierror}
\end{align}
where $\psi_{e}(x_K)$ is a result of averaging
over $N$ sample paths.
Setting $\beta = 1$, we have 
$$M_{B}  =  \left(1+ L^2 \right) (B^2 \slash 4) + 4U^2 =  n(a^2+a^4+4)+3a^4, 
\hspace{1mm} \textrm{and} \hspace{1mm} M_{\nu}  = \left( 1+ L^2 \right)\nu^2 = \left( n + na^2 + 3a^2\right)(n+1)\epsilon^2.$$
\subsection{Almost sure convergence behavior} \label{sec:perf}
In this subsection, we compare the a.s. convergence behavior of the
extragradient and mirror-prox schemes under two different distance metrics. \ask{Table~\ref{tab:exg}} displays
$\psi(x_K)$ generated from the ESA scheme for increasing number of major iterations for
the \uvs{four} problems of interest. We observe that in the
fractional quadratic and nonlinear problems, the ESA scheme performs
relatively well, barring two instances. Notably, much of the
progress is made in the first 1000 iterations.
\begin{table}[htbp]
\begin{center}
{\tiny
\begin{tabular}{|c|c|c|c|c|} \hline
& $n$ & \multicolumn{3}{|c|}{Error $\psi(x_K)$}   \\ \hline
& & $K$ = 1 &  $K$ = 1000  & $K$ = 15000 \\
\hline
\multirow{2}{*}{Frac. Quad.} &10 &6.017e+00  &4.690e-02  &7.951e-04 \\
& 15 &7.473e+00  &1.441e-01  &2.959e-02 \\
\hline
 \multirow{2}{*}{Frac. Nonlin.} & 10 &5.345e+00  &2.754e-02  &2.955e-03 \\
& 15 &7.145e+00  &9.433e-03  &1.288e-02 \\
\hline
\multirow{2}{*}{Nash game}  & 10 &1.581e+01  &4.624e-01  &1.634e-01 \\
& 15 &2.165e+01  &5.613e-01  &2.377e-01 \\
\hline
\multirow{2}{*}{Watson-CP} & 10 &9.695e-01  &2.329e-01  &2.477e-01 \\
& 15 &9.381e-01  &1.357e-01  &1.255e-01 \\
\hline
\end{tabular}}
\caption{{Asymptotics of ESA}}
\label{tab:exg}
\end{center}
\end{table}

%
Next, we compare the stochastic extragradient scheme with two prox-based
generalizations that employ two distance functions proposed by
Nemirowski~\cite{nemirovski04prox} given by
${s}_a (x) = \sum_{i=1}^{n}(x_i+\delta)\log(x_i+\delta)$ and
${s}_b (x) = \log(n) \sum_{i=1}^{n}x_{i}^{\left(1+\frac{1}{\log(n)}\right)}.$
The variants of MPSA, referred to as MPSA-a and MPSA-b respectively, are
studied and the results are compared with the ESA scheme in
\ask{Table~\ref{tab:comparison}} for ten nonlinear fractional problems in the
test set for progressively increasing number of major iterations.  It is
observed  that the ESA scheme sometimes (but not always) performs better
than MPSA-a from an error standpoint but each step of MPSA-a (and
MPSA-b) tends to require more effort as captured by the CPU time.

\begin{table}[htbp]
{\tiny
\begin{center}
\begin{tabular}{|c|c|c|c|c|c|c|c|} \hline
$n$ & $K$ & \multicolumn{2}{|c|}{Projection} & \multicolumn{2}{|c|}{Prox-A}& \multicolumn{2}{|c|}{Prox-B}  \\
\hline
&  & $\psi(x_K)$ & Time (s) & $\psi(x_K)$ & Time(s) & $\psi(x_K)$ & Time (s)\\
\hline
\multirow{2}{*}{10} &1000 & 2.754e-02 & 1.411e+01 & 1.352e-01 & 2.662e+01 & 1.624e-01 & 2.229e+01\\
& 15000 & 2.955e-03 & 2.057e+02 & 1.019e-01 & 3.969e+02 & 8.953e-02 & 4.551e+02\\
\hline
\multirow{2}{*}{15} &1000 & 9.433e-03 & 1.388e+01 & 3.508e-02 & 4.528e+01 & 2.277e-02 & 2.464e+01\\
& 15000 & 1.288e-02 & 2.111e+02 & 1.578e-02 & 6.314e+02 & 1.107e-02 & 8.836e+02\\
\hline
\multirow{2}{*}{19} &1000 & 1.030e-01 & 1.734e+01 & 1.677e-01 & 4.998e+01 & 3.652e-01 & 5.098e+01\\
& 15000 & 8.360e-02 & 2.603e+02 & 1.179e-01 & 8.753e+02 & 2.398e-01 & 9.625e+02\\
\hline
\end{tabular}
\caption{{Comparison of SA schemes for frac. nonlin. problems}}
\label{tab:comparison}
\end{center}
}
\end{table}

\subsection{Error analysis and optimal choices of $\gamma_0$}\label{sec:error}
While the previous results focused on asymptotics, we now compare the
empirical rates with the theoretically predicted rates, as discussed
in Sec.~\ref{sec:rate}. In obtaining the empirical results, the initial
steplength $\gamma_0$ was set to be $(1+\sqrt{33}) \slash (4\sigma)$  and fifteen different sample
paths of ESA were generated to compute $\Psi_e$~\eqref{eq:psierror}. Note that the choice of such a steplength \ka{is to ensure 
that $1  \leq \sigma \gamma_0  \leq 2$} to further demonstrate the alignment with the theoretical rate statement. 
{Given that the expectation may be evaluated, we may solve the original
problem to obtain an estimate of $x^*$.} \ask{Table~\ref{tab:gamma1}} compares the analytical bounds
with empirical results for the given set of problems in increasing
iterations. For the (monotone) problems considered, the theoretical bound is
shown to be valid but relatively weak.
\begin{table}[htbp]
{\tiny
\begin{center}
\begin{tabular}{|c|c|c|c|c|c|c|c|c|c|c|} \hline
Dim ($n$)& \multicolumn{2}{|c|}{$K$ = 1} & \multicolumn{2}{|c|}{$K$ =
	100} & \multicolumn{2}{|c|}{$K$ = 1000} & \multicolumn{2}{|c|}{$K$ =
		10000} & \multicolumn{2}{|c|}{$K$ = 150000} \\ \hline
& $\psi_{e}(x_K)$ & $\psi_{b}(x_K)$ & $\psi_{e}(x_K)$ & $\psi_{b}(x_K)$ & $\psi_{e}(x_K)$ &
$\psi_{b}(x_K)$ & $\psi_{e}(x_K)$ & $\psi_{b}(x_K)$ & $\psi_{e}(x_K)$ & $\psi_{b}(x_K)$\\
\hline
 5 & 3.455e+00 & 6.007e+04 & 1.024e-04 & 6.007e+02 & 4.540e-05 & 6.007e+01 &2.544e-05 & 6.007e+00 & 2.246e-05 & 4.005-01 \\
6 & 4.382e+00 & 1.038e+05 & 3.227e-04 & 1.038e+03 & 5.512e-05 & 1.038e+02 &3.779e-05 & 1.038e+01 & 3.372e-03 & 6.920e-01 \\
7 & 5.324e+00 & 1.648e+05 & 3.323e-04 & 1.648e+03 & 9.332e-05 & 1.648e+02 &5.180e-05 & 1.648e+01 & 5.823e-05 & 1.098e+00 \\
8 & 6.275e+00 & 2.460e+05 & 2.218e-03 & 2.460e+03 & 2.218e-03 & 2.460e+02 &2.218e-03 & 2.460e+01 & 2.218e-03 & 1.640e+00 \\
9 & 7.234e+00 & 3.503e+05 & 5.531e-04 & 3.503e+03 & 1.397e-04 & 3.503e+02 &1.241e-04 & 3.503e+01 & 1.055e-04 & 2.335e+00 \\
10 & 8.197e+00 & 4.806e+05 & 5.201e-03 & 4.806e+03 & 5.201e-03 & 4.806e+02 &5.201e-03 & 4.806e+01 & 5.201e-03 & 3.204e+00 \\
\hline
\end{tabular}
\caption{{Analytical vs Empirical Bounds for stochastic Nash-Cournot Game.}}
\label{tab:gamma1}
\end{center}
}
\end{table}

{We now investigate the benefit of utilizing steplength close to the optimal $\gamma_0$,
denoted by $\gamma_0^*$. We choose $\epsilon = 0.02$ and our steplength is further given
by $\gamma_0^* = ((2-\epsilon) \slash 2\sigma) = 0.99 \slash \sigma$.
Here, we consider the same set of problems as in the previous section and report
the behavior of the proposed extragradient scheme in \ask{Table~\ref{tab:newgam1}} for six different
choices of $\gamma_0$, ranging from $0.0017 \gamma_0^*$ to $170
\gamma_0^*$ in factors of $10$. It can be seen that steplengths close to 
$\gamma_0^*$ perform either the best (or close to the best) for all schemes. In fact,
a poorly chosen steplength leads to significant drop off in
	performance.}
\begin{table}[htbp]
{\tiny
\begin{center}
\begin{tabular}{|c|c|c|c|c||c||c|c|} \hline
& & \multicolumn{6}{|c|}{Empirical error $\psi_e$} \\ \hline
Dim ($n$)&  Iteration $(K)$ & $ 0.0017\gamma_0^{*}$ & $
0.017\gamma_0^{*}$ & $ 0.17\gamma_0^{*}$ & {\bf $1.7\gamma_0^{*}$} &
$17 \gamma_0^{*}$ & $170\gamma_0^{*}$ \\
\hline
\multirow{1}{*}{5}
& 15000 & 2.754e+00& 3.660e-01& 1.748e-05& {\bf 1.716e-05}& 3.084e-05& 9.780e-05\\
\hline
\multirow{1}{*}{10}
& 15000 & 5.406e+00& 1.366e-01& 1.346e-04& {\bf 1.434e-04}& 2.113e-04& 4.360e-04\\
\hline
\multirow{1}{*}{14}
& 15000 & 6.940e+00& 6.410e-02& 9.316e-04& {\bf 9.316e-04}& 9.316e-04& 1.486e-03\\
\hline
\end{tabular}
\caption{{Optimality error for varying choices of $\gamma_0$ \vs{(Bold represents optimal $\gamma_0$)}}}
\label{tab:newgam1}
\end{center}
}
\end{table}

\section{Concluding remarks}\label{sec:concl}
Variational inequality problems represent a useful tool for modeling a
range of phenomena arising in engineering, economics, and the applied
sciences. As the role of uncertainty grows, there has been a growing
interest in the {\em stochastic} variational inequality problem.
However, much of the past research, particularly the algorithmic
aspects, have focused on monotone stochastic variational inequality
problems.  In this context, we provide amongst the first results for claiming a.s.
convergence of the solution iterates to the solution set produced by a
(stochastic) extragradient scheme as well as mirror-prox
generalizations. We also show that similar statements can be provided
for monotone SVIs under a weak-sharpness requirement; notably much of
the prior research for monotone SVIs uses averaging techniques in
showing that the gap function convergence in an expected-value sense.
Under  stronger assumptions on the map, we show that both \vs{the} extragradient
and \vs{the} mirror-prox schemes \vs{attain} the optimal rate of convergence in terms
of solution iterates, rather than in terms of the gap function.
Importanly, we further refine the rate statement by deriving the optimal
initial steplength. Notably, we see a modest degradation of the rate
from strongly monotone SVIs to strongly pseudomonotone SVIs. Preliminary
numerics suggest that the schemes perform well on a breadth of
pseudomonotone and non-monotone problems. Furthermore, empirical
observations suggest  that significant benefit may accrue in
terms of mean-squared error from employing the optimal initial
steplength.  Our work has made an initial step towards understanding how stochastic
approximation schemes can be extended to regimes where
pseudomonotonicity, rather than  monotonicity, of the map holds. Yet, we
believe much remains to be understood regarding  how stochastic
approximation schemes can be extended/modified to contend with far
weaker requirements on the map.
\section{Appendix}
\begin{lemma}
\label{lemm:t0}
\em
Consider the function $t_0(\gamma_0)$ defined as 
$$ t_0(\gamma_0) \triangleq \frac{\gamma_0^2}{2\sigma \gamma_0 - \lfloor
	2\sigma \gamma_0 \rfloor},$$
\uvs{where $\sigma_0$ denotes the strong pseudomonotonicity constant.}
	Then the following hold:
\begin{enumerate}
\item[(a)] A minimizer of $t_0(\gamma_0)$ cannot exist in an interval $\lfloor 2\sigma \gamma_0 \rfloor
 \in [n,n+1]$ where $n > 1$.  
\item[(b)] The infimum of $t_0(\gamma_0)$ is given by the following:
$$ \uvs{ f^* \triangleq } \inf_{\gamma_0}  \left\{t_0(\gamma_0) \mid 1 < \ 2\sigma \gamma_0 \ < 2 \right\}=
\frac{1}{\sigma^2}.$$
\item[(c)] \uvs{Suppose} an $\epsilon$-infimum of $t_0(\gamma_0)$, \uvs{denoted by $f^*_{\epsilon}$,} satisfies $f^*_{\epsilon} \leq f^* + \beta \epsilon$ for some $\beta > 0$. Then $f^*_{\epsilon}$ is achieved by
$\gamma_0 = \frac{2-\epsilon}{2\sigma}$ \uvs{and satisfies $\vs{f_{\epsilon}^*} \leq f^* + \frac{2\epsilon}{\sigma^2}$, } where $\epsilon \in (0,1/2)$.
\end{enumerate}
\end{lemma}
\begin{proof}
\noindent {\bf (a).} We begin by observing that if $2\sigma \gamma_0 \in \mathbb{Z}_+$, then
$t_0(\gamma_0) = +\infty.$ Consequently, any minimizer of
	$t_0(\gamma_0)$ has to satisfy $2\gamma_0 \sigma \not \in
	\mathbb{Z}_+$. We proceed to show that $t_0(\gamma_0)$ does not admit
	a minimizer in $(n,n+1)$ where $n > 1$.
 Assume this is  false and suppose there exists a minimizer $\gamma_0^*$ \uvs{satisfying} $2\gamma_0^*\sigma \in
(n,n+1)$. But  there exists  a $\tilde \gamma_0$ such that 
$2\sigma \tilde \gamma_0 \in (n-1,n)$. 
In fact, $t_0({\tilde \gamma}_0) < t_0(\gamma_0^*)$ as we show next and our claim follows. 
$$ t_0(\tilde \gamma_0)  = \left( \frac{\tilde \gamma_0^2}{2\sigma
		\tilde \gamma_0 - \lfloor 2 \sigma \tilde \gamma_0 \rfloor}
		\right) <  \left( \frac{(\gamma^*_0)^2}{2\sigma
		\gamma^*_0 - \lfloor 2 \sigma \gamma^*_0 \rfloor}
		\right) = t_0(\gamma_0^*).$$
\noindent {\bf (b).} \uvs{From (a), 
it follows that} if a minimizer exists, it has to satisfy $2\gamma_0^*
			\sigma \in (1,2)$. It follows that $t_0(\gamma_0)$ reduces
			to $\gamma_0^2/(2\sigma \gamma_0 - 1).$ 
			Consider the following optimization
			problem:
	$$ \inf_{ \gamma_0} \left\{ \frac{\gamma_0^2}{(2\sigma \gamma_0 -
			1)} \mid 1 < 2\sigma \gamma_0
	< 2 \right\}. $$
	We observe that $t_0(\gamma_0)$ is a strictly decreasing function by
	noting that 
	$$ t'_0(\gamma_0) = \frac{2\gamma_0}{2\sigma \gamma_0-1} -\frac{ 2\sigma
	\gamma_0^2 }{(2\sigma \gamma_0 -1)^2} = \frac{2\gamma_0}{2\sigma
		\gamma_0-1} \left(1-\frac{\sigma \gamma_0}{2\sigma \gamma_0-1}
				\right) = \frac{2\gamma_0}{2\sigma
		\gamma_0-1} \left(\frac{\sigma \gamma_0-1}{2\sigma \gamma_0-1}
				\right) < 0,$$
	since $\sigma \gamma_0 < 1$. It follows that the infimum is at the
	end-point given by $\sigma \gamma_0 = 1$ implying that 
	$$ \inf_{ \gamma_0} \left\{ \frac{\gamma_0^2}{(2\sigma \gamma_0 -
			1)} \mid 1 < 2\sigma \gamma_0
	< 2 \right\} = \frac{1}{\sigma^2}. $$
\noindent {\bf (c.)} Suppose $\widehat{\gamma}_0 =
\frac{2-\epsilon}{2\sigma}$ where $\epsilon \in (0,1/2)$. Then we have that 
\begin{align*}
f^*_{\epsilon} - f^* = 
\frac{(2-\epsilon)^2}{4\sigma^2 (1-\epsilon)} - \frac{1}{\sigma^2} \leq
\frac{1}{\sigma^2} \left(\frac{1}{1-\epsilon} - 1\right) 
			= 	\frac{1}{\sigma^2}\frac{\epsilon}{1-\epsilon} 
			= \frac{2}{\sigma^2} \epsilon. 
\end{align*}
It follows that $\uvs{f^*_{\epsilon} \leq f^* + \frac{2\epsilon}{\sigma^2}}.$
\end{proof}
\begin{example} Unfortunately, while one can derive an infimum of the
above discontinuous optimization problem, this infimum cannot be
achieved and the problem lacks a minimizer as proved in the above result. Yet, this infimum is
informative in developing an approximate $\epsilon$-solution as part (c)
	shows. We proceed to use this $\epsilon$-infimum in deriving rate
	statements and demonstrate this result through an example.  Suppose $\sigma = 0.1\sqrt{1.3}$. Then $t_0(\gamma_0)$
is shown as a solid line with discontinuities in Fig.~\ref{t0fig} while the
dashed flat line displays the infimum $1/\sigma^2$. 
\begin{figure}
\centering
\includegraphics[width=4in]{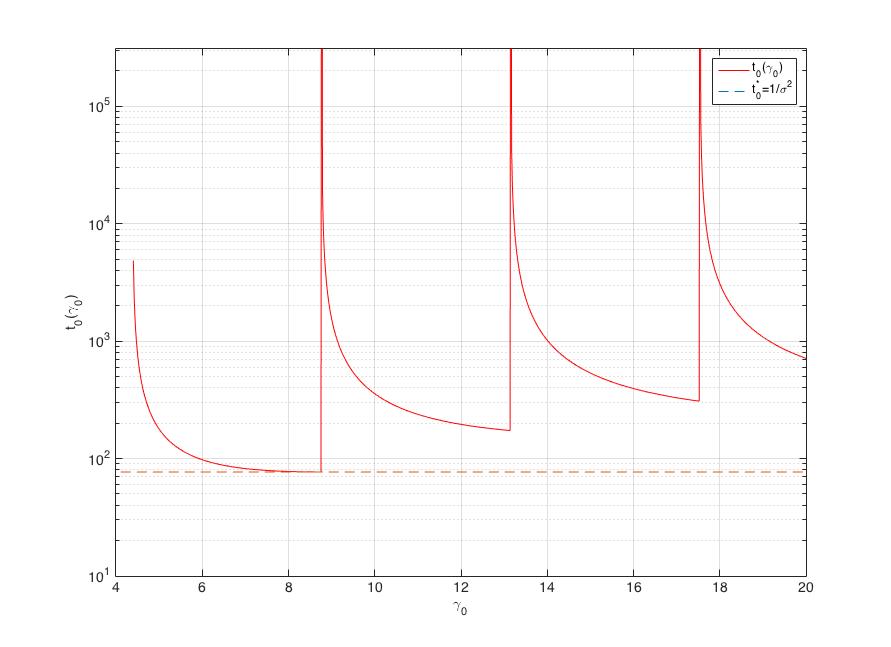}
\caption{A schematic of $t_0(\gamma_0)$}
\label{t0fig}
\end{figure}
\end{example}

{\begin{lemma}\label{ind_lemma}
Consider the following recursion:
$ a_{k+1} \leq (1-2c\theta/k) a_k + \half \theta^2 M^2/k^2, $
where $\theta$ and $M$ are positive constants, $a_k \geq 0$, and $(1-2c\theta) < 0$.
Then for $k \geq 1$,  we have that
$$2a_k \leq \frac{\max\left(\frac{\theta^2}{2c\theta - \lfloor
		2c\theta\rfloor} M^2,2a_1\right)}{k}.$$
\end{lemma}
\begin{proof}
We begin by noting that $\bar \epsilon > 0$ and $\kappa >1$ as seen
	next.
$$\kappa = \left( 1+\frac{\bar{k}-1}{2c\theta-\bar{k}}\right) =
\left( \frac{2c\theta-1}{2c\theta-\lfloor 2c\theta \rfloor }\right) > 1. $$
We consider the following cases for $k$.
\be
\item[]Case 1: Consider $k = 1$. Then the following holds: $a_{2} \leq (1-2c\theta)a_1 + \frac{1}{2}\theta^2 M^2.$ If $(2c\theta-1) > 0$, we may rearrange the inequalities as follows:
\begin{align*}
(2c\theta-1)a_1 & \leq -a_2 + \frac{1}{2} \theta^2M^2 \leq \frac{1}{2} \theta^2M^2 \hspace{2mm} \textrm{or} \hspace{2mm} 2a_1 \leq \frac{1}{2c\theta-1} \theta^2M^2 \\
 \Longrightarrow 2a_1 & \leq \max \left( (2c\theta-1)^{-1}\theta^2M^2, 2a_1\right) \\
 & \leq \max \left( (2c\theta-1)^{-1}\theta^2M^2\kappa,2a_1 \right),
\end{align*}
where $\kappa > 1$.
\item[]Case 2: $1  \ < \  k \leq \bar{k}$. Recall that when $k \leq \bar{k}$, we have that
$$(1-2c\theta \slash k) \leq (1-2c\theta \slash \bar{k}) = \left( 1-(2c\theta) \slash (\lceil 2c\theta\rceil)\right) < 0.$$
Then the following holds:
\begin{align*}
a_{k+1} & \leq \left(\left(1-\frac{2c\theta}{k}\right)a_{k} + \frac{1}{2}\frac{\theta^2 M^2}{k^2} \right) \\
a_k \left( \frac{2c\theta}{k}-1\right) & \leq -a_{k+1} + \frac{\theta^2M^2}{2k^2}  \leq \frac{\theta^2 M^2}{2k^2}\\
\Longrightarrow a_k & \leq \frac{\theta^2M^2}{2k^2}\left(
		\frac{k}{2c\theta-k}\right) = \frac{\theta^2M^2}{2k(2c\theta-k)}.
\end{align*}
By the definition of $\kappa$ and $u \triangleq \max \left(\frac{\kappa \theta^2 M^2}{(2c\theta-1)},
		{2a_1}\right)$, we may conclude the following:
\begin{align*}
2a_{k} \leq \frac{\theta^2M^2}{k(2c\theta-k)} & =
	\frac{\theta^2M^2}{k(2c\theta-1)}\left(\frac{2c\theta-1}{2c\theta-k}\right) \\
& \leq
\frac{\theta^2M^2}{k(2c\theta-1)}\left(\frac{2c\theta-k+\bar{k}-1}{2c\theta-k}\right) \\
& \leq \frac{\theta^2M^2}{k(2c\theta-1)} \left(
		1+\frac{\bar{k}-1}{2c\theta-k}\right) \\
&  \leq \frac{\theta^2M^2}{k(2c\theta-1)} \left(
		1+\frac{\bar{k}-1}{2c\theta-\bar k }\right) = \frac{\theta^2M^2}{k(2c\theta-1)} \kappa \\
& \leq {1\over k}\max \left(\frac{\kappa \theta^2 M^2}{(2c\theta-1)},
		{2a_1}\right) = \frac{u}{k},
\end{align*}
\item[]Case 3: $k > \bar{k}$. Suppose, this holds for $k > \bar{k}$, implying that $2a_k \leq \frac{\max(\theta^2M^2(2c\theta-1)^{-1}\kappa, 2a_1)}{k}$.
We proceed to show that this holds for $k: = k+1$ where $u \triangleq \max \left( \theta^2M^2 (2c\theta-1)^{-1}\kappa, 2a_1 \right)$ and $(1-\frac{2c\theta}{k}) > 0$ since $k > \bar{k}$:
\begin{align*}
a_{k+1} & \leq \left( 1-\frac{2c\theta}{k}\right) \frac{u}{2k} + \frac{1}{2}\left( \frac{\theta^2M^2}{k^2}\right) \\
& = \left( 1-\frac{2c\theta}{k}\right)\frac{u}{2k} + \frac{(2c\theta-1)}{2k} \left( \frac{\theta^2M^2}{(2c\theta-1)k}\right) \\
&\leq \left( 1-\frac{2c\theta}{k}\right) \frac{u}{2k} +
\frac{(2c\theta-1)}{2k} \left( \frac{\theta^2M^2\kappa}{(2c\theta-1)k}\right) \\
& \leq \left(\frac{u}{2k}-\left(\frac{2c\theta}{k}\right) \frac{u}{2k}\right) +
\left(\frac{2c\theta-1}{2k}\right) \frac{u}{k} \\
& = \frac{u}{2k}-\frac{1}{k}\left( \frac{u}{2k}\right)  \leq \frac{u}{2k} -\frac{1}{k+1}\left( \frac{u}{2k}\right)  = \frac{u}{2(k+1)}.
\end{align*}
\ee
\end{proof}
}
\begin{lemma} \label{lemm:productfns}
Consider the following problem:
	$\min  \left\{ h(\gamma_0) g(z) \mid
	 \gamma_0 \in \Gamma_0, z \in {\cal Z}\right\},$
where $h$ and $g$ are positive functions over $\Gamma_0$ and $\cal Z$,
	  respectively. If $\bar \gamma_0$ and $\bar z$ denote global minimizers of
$h(\gamma_0)$ and $g(z)$ over $\Gamma_0$ and ${\cal Z}$, respectively,
	then  the following holds:
$$\ask{\min_{\gamma_0 \in \Gamma_0, z \in \cal Z} h(\gamma_0) g(z) = h(\bar	 \gamma_0)  g(\bar z).}$$
\end{lemma}
\begin{proof} The proof has two steps. First,  we note that
$ \min_{\gamma_0 \in \Gamma_0, z \in \cal Z} h(\gamma_0) g(z) \geq h(\bar
		\gamma_0)  g(\bar z), $
implying that at any global minimizer $(\gamma_0^*,z^*)$,
		 \begin{align} \label{e1}
		 h(\gamma_0^*)g(z^*) \geq h(\bar
		\gamma_0)  g(\bar z).\end{align}
Second, since $(\bar \gamma_0, \bar z) \in \Gamma_0 \times \cal Z$,
	we have that $h(\gamma_0^*)g(z^*)$ has an optimal value that is no
	smaller than that the value associated with any feasible solution or
\begin{align}
\label{e2} h(\gamma_0^*) g(z^*) = \min_{\gamma_0 \in \Gamma_0, z \in \cal Z}
h(\gamma_0) g(z)  \leq h(\bar \gamma_0) g(\bar z). \end{align}
By combining \eqref{e1} and \eqref{e2}, the result follows.
\end{proof}
\section*{Acknowledgements} The authors are grateful to Dr. Farzad Yousefian for his valuable suggestions on a previous version. We particularly appreciate the comments of \vs{the referees and the editor, all of which  have led to significant improvements in the manuscript.} 
\bibliographystyle{ieeetran}

\def\cprime{$'$} \def\cprime{$'$}

\end{document}